\documentclass[12pt,british,english]{article}
\usepackage[T1]{fontenc}
\usepackage[latin9]{inputenc}
\usepackage{geometry}
\usepackage{babel}
\usepackage{refstyle}
\usepackage{amsmath}
\usepackage{amsthm}
\usepackage{amssymb}
\usepackage{stmaryrd}
\usepackage{graphicx}
\usepackage[unicode=true,pdfusetitle,
 bookmarks=true,bookmarksnumbered=false,bookmarksopen=false,
 breaklinks=false,pdfborder={0 0 1},backref=false,colorlinks=false]
 {hyperref}

\makeatletter

\AtBeginDocument{\providecommand\Thmref[1]{\ref{Thm:#1}}}
\AtBeginDocument{\providecommand\secref[1]{\ref{sec:#1}}}
\AtBeginDocument{\providecommand\thmref[1]{\ref{thm:#1}}}
\AtBeginDocument{\providecommand\Secref[1]{\ref{Sec:#1}}}
\AtBeginDocument{\providecommand\lemref[1]{\ref{lem:#1}}}
\AtBeginDocument{\providecommand\propref[1]{\ref{prop:#1}}}
\AtBeginDocument{\providecommand\defref[1]{\ref{def:#1}}}
\AtBeginDocument{\providecommand\corref[1]{\ref{cor:#1}}}
\AtBeginDocument{\providecommand\factref[1]{\ref{fact:#1}}}
\AtBeginDocument{\providecommand\subsecref[1]{\ref{subsec:#1}}}
\AtBeginDocument{\providecommand\figref[1]{\ref{fig:#1}}}

\RS@ifundefined{subsecref}
  {\newref{subsec}{name = \RSsectxt}}
  {}
\RS@ifundefined{thmref}
  {\def\RSthmtxt{theorem~}\newref{thm}{name = \RSthmtxt}}
  {}
\RS@ifundefined{lemref}
  {\def\RSlemtxt{lemma~}\newref{lem}{name = \RSlemtxt}}
  {}

\theoremstyle{plain}
\newtheorem{thm}{\protect\theoremname}[section]
\theoremstyle{definition}
\newtheorem{defn}[thm]{\protect\definitionname}
\theoremstyle{remark}
\newtheorem{rem}[thm]{\protect\remarkname}
\theoremstyle{plain}
\newtheorem{prop}[thm]{\protect\propositionname}
\theoremstyle{plain}
\newtheorem{fact}[thm]{\protect\factname}
\theoremstyle{plain}
\newtheorem{lem}[thm]{\protect\lemmaname}
\theoremstyle{plain}
\newtheorem{cor}[thm]{\protect\corollaryname}
\theoremstyle{remark}
\newtheorem*{rem*}{\protect\remarkname}
\theoremstyle{definition}
\newtheorem*{defn*}{\protect\definitionname}
\theoremstyle{remark}
\newtheorem*{claim*}{\protect\claimname}
\theoremstyle{definition}
\newtheorem{problem}[thm]{\protect\problemname}

\usepackage{indentfirst}
\usepackage{algorithm}
\makeatletter
\renewcommand{\ALG@name}{Code Snippet}
\makeatother

\usepackage{xcolor}
\definecolor{keyword}{HTML}{BA2CA3}
\definecolor{string}{HTML}{D12F1B}
\definecolor{comment}{HTML}{008400}
\definecolor{identifier}{HTML}{0B4F79}

\newref{def}{refcmd={\hyperref[#1]{Definition \ref{#1}}}}
\newref{prop}{refcmd={\hyperref[#1]{Proposition \ref{#1}}}}
\newref{cor}{refcmd={\hyperref[#1]{Corollary \ref{#1}}}}
\newref{lem}{refcmd={\hyperref[#1]{Lemma \ref{#1}}}}
\newref{fig}{refcmd={\hyperref[#1]{Figure \ref{#1}}}}
\newref{def}{refcmd={\hyperref[#1]{Definition \ref{#1}}}}

\makeatother

\addto\captionsbritish{\renewcommand{\claimname}{Claim}}
\addto\captionsbritish{\renewcommand{\corollaryname}{Corollary}}
\addto\captionsbritish{\renewcommand{\definitionname}{Definition}}
\addto\captionsbritish{\renewcommand{\factname}{Fact}}
\addto\captionsbritish{\renewcommand{\lemmaname}{Lemma}}
\addto\captionsbritish{\renewcommand{\problemname}{Problem}}
\addto\captionsbritish{\renewcommand{\propositionname}{Proposition}}
\addto\captionsbritish{\renewcommand{\remarkname}{Remark}}
\addto\captionsbritish{\renewcommand{\theoremname}{Theorem}}
\addto\captionsenglish{\renewcommand{\claimname}{Claim}}
\addto\captionsenglish{\renewcommand{\corollaryname}{Corollary}}
\addto\captionsenglish{\renewcommand{\definitionname}{Definition}}
\addto\captionsenglish{\renewcommand{\factname}{Fact}}
\addto\captionsenglish{\renewcommand{\lemmaname}{Lemma}}
\addto\captionsenglish{\renewcommand{\problemname}{Problem}}
\addto\captionsenglish{\renewcommand{\propositionname}{Proposition}}
\addto\captionsenglish{\renewcommand{\remarkname}{Remark}}
\addto\captionsenglish{\renewcommand{\theoremname}{Theorem}}
\providecommand{\claimname}{Claim}
\providecommand{\corollaryname}{Corollary}
\providecommand{\definitionname}{Definition}
\providecommand{\factname}{Fact}
\providecommand{\lemmaname}{Lemma}
\providecommand{\problemname}{Problem}
\providecommand{\propositionname}{Proposition}
\providecommand{\remarkname}{Remark}
\providecommand{\theoremname}{Theorem}

\begin{document}
\title{Graphs with queue number three and unbounded stack number}
\author{Yui Hin Arvin Leung\thanks{Department of Pure Mathematics and Mathematical Statistics, University
of Cambridge. Email: yhal2@cam.ac.uk}}
\maketitle
\begin{abstract}
We prove that the graphs $T\boxslash P$ have unbounded stack number
and queue number $3$, where $T$ is a tree and $P$ is a path, and
$\boxslash$ denotes the graph strong product but with one of the
directions removed. The previous best known results is that graphs
with queue number $4$ can have unbounded stack number.

\tableofcontents{}
\end{abstract}

\section{Introduction}

Firstly, we shall define the stack and queue number of a graph. 
\begin{defn}
Let $G$ be a graph, and $L$ be a linear order on $V(G)$. We say
edges $e_{1}$ and $e_{2}$ are \emph{separated} (resp. \emph{nest},
\emph{cross}) with respect to $L$ if:
\begin{itemize}
\item Separated, if $a<b<c<d$ or $c<d<a<b$
\item Nest, if $a<c<d<b$ or $c<a<b<d$ 
\item Cross, if $a<c<b<d$ or $c<a<d<b$ 
\end{itemize}
after uniquely writing $e_{1}=(a,b)$ , $e_{2}=(c,d)$ , $a<b$ and
$c<d$. All comparisons are with respect to $L$. If $e_{1},e_{2}$
share any endpoints, they do not form a separated, a nesting, nor
a crossing pair.

\begin{figure}[h]
\begin{centering}
\includegraphics{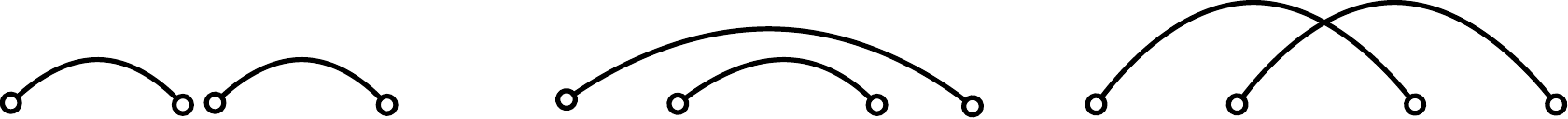}
\par\end{centering}
\caption{From left to right: Disjoint, Nest and Cross}
\end{figure}
\end{defn}

\begin{defn}
Let $G$ be a graph. A \emph{$k$-queue layout} $(L,\phi)$ consists
of a total order $L$ on $V(G)$, and $\phi$ a $k$- colouring of
$E(G)$, such that any two edges of the same colour do not nest with
respect to $L$.

A \emph{$k$-stack layout} $(L,\phi)$ consists of a total order $L$
on $V(G)$, and $\phi$ a $k$-colouring of $E(G)$, such that any
two edges of the same colour do not cross with respect to $L$.

The \emph{queue number} of $G$ is the smallest integer $k$ such
that $G$ admits a $k$-queue layout. The \emph{stack number} of $G$
is the smallest integer $k$ such that $G$ admits a $k$-stack layout. 
\end{defn}

Two natural questions have been asked regarding stack and queue numbers
\begin{itemize}
\item Is stack-number bounded by queue-number?
\item Is queue-number bounded by stack-number?
\end{itemize}
In \cite{key-first-4}, MLA Dujmovi\'{c}, Vida, et al. showed the
graphs $S\square(P\boxslash P)$ have unbounded stack number with
queue number $4$, where $S$ is a star and $P$ are paths. This resolved
the first question. In \cite{three-products}, Eppstein, David, et
al. showed that the graphs $P\boxslash P\boxslash P$ also have a
queue number of $4$ with unbounded stack number, giving another example.
The $\boxslash$ denotes the following graph product variant:
\begin{defn}
Let $G,H$ be two undirected graph. Let $G',H'$ be orientations of
$G,H$. The graph $G\boxslash H$ is defined to be the undirected
graph with vertex set $\{(g,h)|g\in V(G),h\in V(H)\}$, and three
types of edges 
\begin{itemize}
\item $(g,h)\sim(g',h)$ for $(g\sim g')\in G$, $h\in H$
\item $(g',h)\sim(g,h)$ for $(h\sim h')\in H,g\in G$.
\item $(g,h)\sim(g',h')$ for $g\rightarrow g'\in G'$, $h\rightarrow h'\in H'$ 
\end{itemize}
\end{defn}

Our key contribution is the following:
\begin{thm}
\label{thm:main-result-statement}For every $s\in\mathbb{N}$, there
exist a graph $G$ with queue number at most $3$ , and stack number
at least $s$. 
\end{thm}

We will take the graph $G$ to be of form $T\boxslash P$, where $T$
is a tree, $P$ is a path. Write $P$ as the path $1-2-...-m$. In
defining $T\boxslash P$, we take $T$ to be directed towards the
root, and $P$ to be directed away from $1$. In other words, the
graph $T\boxslash P$ are defined as follows:
\begin{defn}
Let $T$ be a rooted tree. Let $P$ be a path of length $m$ with
vertex set $\{1,2,...,m\}$. The graph $T\boxslash P$ are defined
as $V(T\boxslash P)=\{(v,i)|v\in V(T),i\in V(P)\}$, with three types
of edges:
\begin{itemize}
\item (Vertical edges) $(v,i)\sim(p,i)$ for where $p$ is a parent of $v$in
$T$ 
\item (Diagonal edges) $(v,i)\sim(p,i+1)$ for where $p$ is a parent of
$v$ in $T$ 
\item (Horizontal edges) $(v,i)\sim(v,i+1)$ for $i\in\{1,2,\cdots,m-1\}$
\end{itemize}
\end{defn}

This paper shows:
\begin{thm}
For every $s\in\mathbb{N}$, there exist a tree $T$ and path $P$
such that $T\boxslash P$ have stack number exceeding $s$.
\end{thm}

\Thmref{main-result-statement} follows as it is not difficult to
see that $T\boxslash P$ have queue number at most $3$ (Proof is
given in \secref{queue-3}).

Both \cite{key-first-4,three-products} asked whether graphs with
queue number $3$ can have unbounded stack numbers. \thmref{main-result-statement}
shows that the answer is yes. In \cite{key-first-4} MLA Dujmovi\'{c},
Vida, et al. asked whether $T\boxtimes P$ have unbounded stack number,
and since $T\boxslash P$ is a subgraph of $T\boxtimes P$ , our result
implies $T\boxtimes P$ also have unbounded stack number.

In the proof of the main theorem, we have applied the high dimensional
Erd\H{o}s--Szekeres theorem (\cite{erdosgoodbound}) many times,
and thus the required size of $T\boxslash P$ grows unreasonably quickly.
As a result, we did not compute the exact size of $T\boxslash P$
needed.

\begin{figure}[h]
\begin{centering}
\includegraphics[scale=0.3]{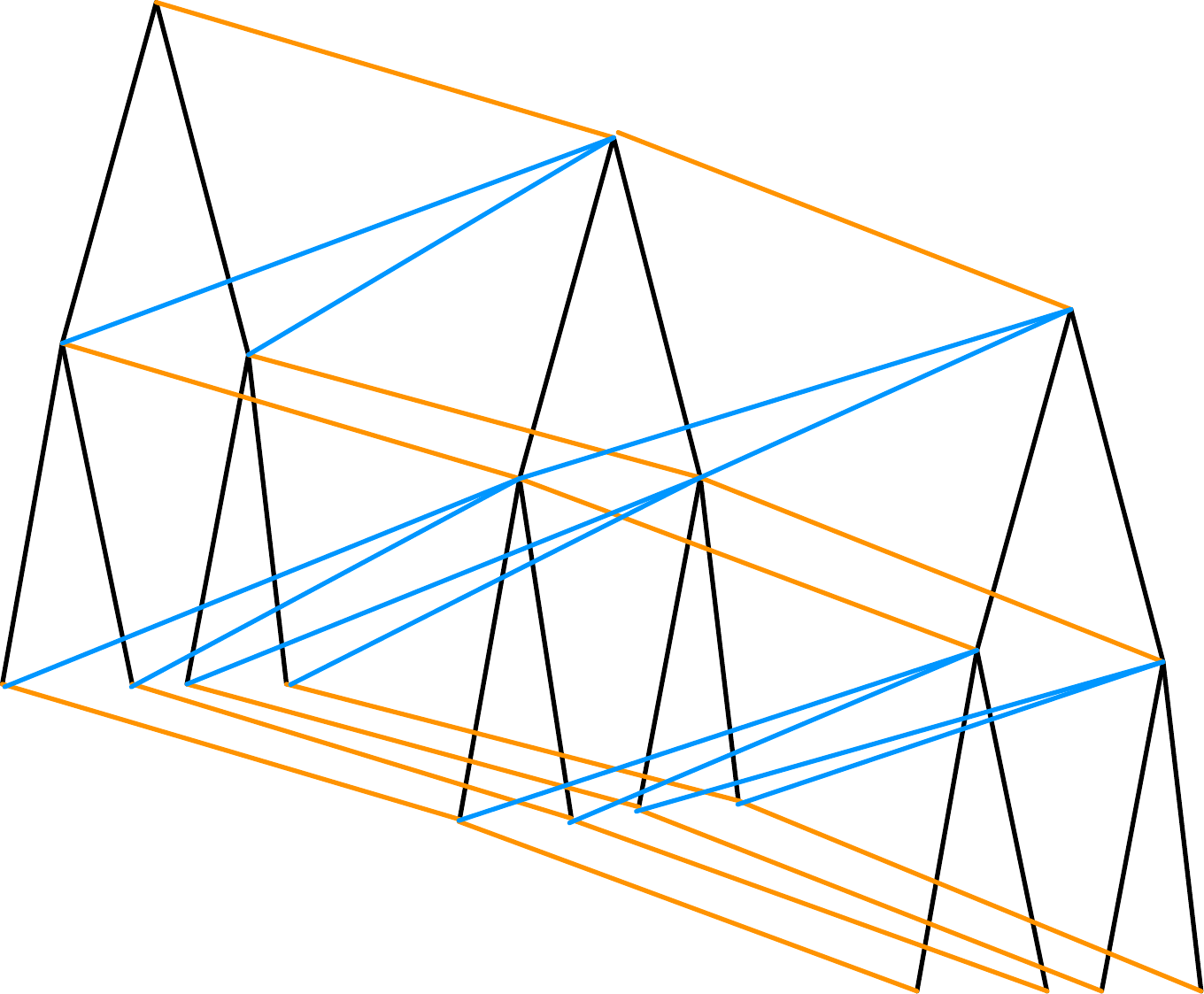}
\par\end{centering}
\caption{$T\boxslash P_{3}$ where $T$ is a binary tree of height $2$. Black
edges are vertical, Blue edges are diagonal, orange edges are Horizontal. }
\end{figure}

\subsection{Overview of the proof }

The proof of the main result starts by fixing a $s-$stack layout,
and derive a contradiction. After fixing a $s$-colouring, we make
use of the non-crossing assumption via monotone sequences. \Secref{sequences}
develops a few lemma about monotone sequences, with \lemref{ultra-inter}
and its variants being the key piece. The first step is found in \secref{ramsey-passes},
where we apply a total of three Ramsey arguments on $T\boxslash P$.
Most of the Ramsey arguments can be applied without properties of
a stack layout, but we need \lemref{ultra-inter} to say more about
the lex-monotone property (\lemref{identity-permutation}). Regarding
assumptions on $L$, lex-monotone is the best we can do by Ramsey
arguments alone, and we cannot make further assumptions about whether
various sequences are monotone increasing or decreasing. The general
monotone directions of various sequences in $T\boxslash P$ is encoded
in a function $Z(i,j,p)$. We should see that $Z(i,j,p)$ are not
completely arbitrary in \lemref{z-consistent}. In \secref{hexgaonal},
we introduce the hexagonal grid, making sense of the condition in
\lemref{z-consistent}. Finally, in \secref{main-proof} we use the
property of the hexagonal grid, as well as \lemref{ultra-inter-weak}
to deduce the contradiction. 

\subsection{\label{subsec:comparison}Comparison with the case of $S\square(P\boxslash P)$}

In \cite{key-first-4}, MLA Dujmovi\'{c}, Vida, et al. showed the
graphs $S\square(P\boxslash P)$ have unbounded stack number, establishing
that graphs with queue number $4$ can have unbounded stack number.
Write $R$ to be the set of leaves in $S$. Write $r$ to be the root
of $S$. We outline their proof using our notations:
\begin{itemize}
\item Apply Ramsey arguments to $S\square(P\boxslash P)$. By keeping only
a subset of vertex in $S$, we can assume edge colours and relative
order are as equal as possible (this is \propref{pass-colour-weak}
and \propref{pass-order-weak} in the case of $T$ have height $1$.)
\item Using Erd\H{o}s--Szekeres theorem, and keeping only a further subset
of $S$, assume that the sequences $A_{p,q}=\{(v,p,q)\}_{v\in R}$
are monotone 
\item Define a 2-colouring on $P_{n}\boxslash P_{n}$ by $Z(p,q)\in\{INC,DEC\}$
, determined by whether the sequence $A_{p,q}$ is monotone increasing/decreasing
respectively. The hex lemma (\lemref{hex-lemma}) shows that we can
find a monochromatic path $P^{*}\subset V(P\boxslash P)$ of length
at least $n$.
\item For each $(a,b)\in P^{*}$, the sequence $A_{p,q}$ are fanned (\defref{fan})
to $(r,p,q)$. Every pairs of adjacent $(x,y)\in P^{*}$ are bundled
(by \propref{bundled}). For every two $x,y\in P^{*}$, their corresponding
sequences must have a large interleave (by \corref{half_inter}).
Pick two vertex on $P^{*}$ which corresponds to fans of the same
colour, \propref{fan-fan} shows that there is a contradiction. 
\end{itemize}
The key differences are as follows, with the important new ideas mentioned: 
\begin{itemize}
\item The behavior of $v\in T\boxslash P$ for different layers of $T$
behave very differently. Instead of three coordinates as in $S\square(P\boxslash P)$,
$T\boxslash P$ should be regarded as a high dimensional structure.
The Ramsey arguments on edge-colours and relative order still applies,
though requires extra notations to make sense of it for high dimensional
case (\propref{pass-colour} and \propref{pass-order})
\item For the second step, we can certainly assume each layer are separately
monotone. However, for the key lemma \lemref{z-consistent}, it is
necessary to assume more relations across layers. The new concept
we used are lex-monotone, obtained by applying a high dimensional
generalization of Erd\H{o}s--Szekeres theorem \cite{erdosgoodbound}.
\item We find that fanned sequences and bundled pairs alone, unlike the
case of $S\square(P\boxslash P)$, are not enough to derive a contradiction.
We make use of an additional type of sequences pairs -- rainbow sequences
(\propref{rainbow}), which arrises in a similar way with bundled
pairs. Properties of rainbow pairs include \propref{rainbow_inter}
and \propref{fan-rainbow}.
\item Instead of a two-dimensional $2$-colouring $Z(p,q)$, in $T\boxslash P$
we can only assume a three-dimensional $2$-colouring $Z(i,p,q)$
(defined in \ref{def:Z}). Fortunately, the colouring $Z(i,p,q)$
are not completely arbitrary and satisfy \lemref{z-consistent}, proved
using \lemref{identity-permutation}. This property can be put into
a nice way using chromatic boundary (\defref{chromatic-boundary}
and \corref{chromatic-perserve})
\item In $S\square(P\boxslash P)$, the sequences corresponding to colours
$Z(p,q)$ are always fanned. As two fans always cross (\propref{fan-fan}),
a contradiction follows easily from a monochromatic path. In our case,
only points $Z(i,i,p)$ corresponds to fans, and points $Z(i,j,p)$
for $j>i$ only correspond to another sequence pair (rainbow or bundled).
The values $Z(i,i,p)$ can be interpreted as the top layer of $T\boxslash P$,
and essentially the top row of a hexagonal grid. We can only derive
contradiction from two interleaving fans, or interleaving fan and
rainbow, but not from two rainbows or with any bundled sequences.
We certainly cannot assume the existence of a long monochromatic path
on the top of the hexagonal grid. This caused considerable difficulties.
Instead of just taking a monochromatic path, we have used the observations
behind the hex lemma as well as additional geometrical properties
to conclude the proof (in \lemref{long-line-bad} and \lemref{top-or-long}).
\end{itemize}

\section{Monotone sequences \label{sec:sequences}}

We make use of the non-crossing assumption through monotone sequences.
Plenty of monotone sequences can be found after applying Ramsey arguments
to $T\boxslash P$. This section discusses only the implications of
non-crossing edges on a linear order, and does not use properties
of $T\boxslash P$.

We begin by assuming $T\boxslash P$ have a $s$-stack layout, thus
a linear order $L$ on $V(T\boxslash P)$ and a $s$-colouring $\phi$
on $E(T\boxslash P)$. $L$ is the only linear order being considered
throughout this section. By sequences, we mean a sequence of vertex
of $T\boxslash P$, but we shall refer to them simply as ``elements''
throughout this section. For $a,b\in V(T\boxslash P)$, we write $a<b$
to mean $a<_{L}b$. An edge refers to an actual edge on $T\boxslash P$,
but through out this section we may simply consider it as an unordered
pair of elements. The colour of an edge is the colour given by $\phi$.
Since $T\boxslash P$ is a stack-layout, edges of the same colour
cannot cross (with respect to $L$).

We use capital letters $A,B,\cdots$ to denote sequences. Their corresponding
lower case letters (e.g. $a_{1},a_{2},\cdots a_{n}$) denote elements
of the sequence. 
\begin{rem}
For the rest of the article, whenever we mention multiple sequences
at the same time, we always assume all elements involved are distinct.
\end{rem}

After making Ramsey arguments on $T\boxslash P$ in \secref{ramsey-passes},
we will find many pairs of sequences satisfying following three assumptions.
This is our starting point.
\begin{defn}
\label{def:basic-monotone}A sequence $A=[a_{1},a_{2}\cdots a_{t}]$
is \emph{monotone increasing} if $a_{1}<a_{2}\cdots<a_{t}$, is \emph{monotone
decreasing} if $a_{1}>a_{2}>\cdots>a_{t}$. $A$ is \emph{monotone}\textbf{
}if it is either monotone increasing or monotone decreasing. 
\end{defn}

\begin{defn}
\label{def:order-consistent}We say two monotone sequences $A,B$
of the same length to be \emph{order consistent}\textbf{ }if either
$A_{i}<B_{i}$ for all $i$ , or $A_{i}>B_{i}$ for all $i$.
\end{defn}

\begin{defn}
\label{def:uniformly-adjacent}We say two monotone sequences $A=[a_{1},a_{2}\cdots a_{t}],B=[b_{1},b_{2}\cdots b_{t}]$
of the same length are \emph{uniformly adjacent} of colour $c$, if
there is an edge of $a_{i}\sim b_{i}$ of colour $c$ for each $i\in\{1,..,t\}$.
Consequently, the edges $a_{i}\sim b_{i}$ do not cross with one another.
\end{defn}

We give a name to pairs of sequences that satisfy the three conditions
simultaneously.
\begin{defn}
We say two sequences $A,B$ of the same length are \textbf{related}
if they are order consistent, uniformly adjacent of some colour, and
are both monotone. They are related of colour $c$ if they are uniformly
adjacent of colour $c$. 
\end{defn}

Two monotone sequences $A,B$ can be in same direction, if they are
both increasing or both decreasing, or in different directions, in
the case that exactly one is increasing. We denote two related sequences
$A,B$ by $A\sim B$, both in the case of same or different directions.
The two cases produce vastly different structures c.f. \propref{bundled}
\propref{rainbow}. 
\begin{defn}
\label{def:strong-interleave}Two monotone sequences $A,B$ of the
same length $k$ are said to \emph{strongly-interleave,} if they are
in the same direction, and one of the following four cases apply:
\begin{itemize}
\item (Both increasing), $a_{1}<b_{1}<a_{2}<b_{2}\cdots<a_{k}<b_{k}$ ,
or $b_{1}<a_{1}<b_{2}<a_{2}\cdots<b_{k}<a_{k}$, or
\item (Both decreasing) $a_{1}>b_{1}>a_{2}>b_{2}\cdots>a_{k}>b_{k}$ , or
$b_{1}>a_{1}>b_{2}>a_{2}\cdots>b_{k}>a_{k}$. 
\end{itemize}
\end{defn}

\begin{prop}
\label{prop:bundled}Let $A,B$ be related sequences. Suppose $A$
and $B$ are in the same direction, then $A$ and $B$ strongly interleave. 
\end{prop}

\begin{proof}
Without loss of generality, we assume $A,B$ are both monotone increasing,
and $a_{i}<b_{i}$ for all $i$(using order consistent assumption).
We can reverse the numberings on both $A,B$, or swap $A$ and $B$
to achieve this. 

For each $i$ , consider the four elements $a_{i},b_{i},a_{i+1},b_{i+1}$.
Then since $a_{i}<b_{i}$ and $a_{i+1}<b_{i+1}$, $a_{i}<a_{i+1}$
and $b_{i}<b_{i+1}$, we know $a_{i}$ must be the smallest and $b_{i+1}$
must be the largest among the four, and so there are two options:
\begin{enumerate}
\item $a_{i}<a_{i+1}<b_{i}<b_{i+1}$
\item $a_{i}<b_{i}<a_{i+1}<b_{i+1}$
\end{enumerate}
The first case has edges $(a_{i},b_{i})$ and $(a_{i+1},b_{i+1})$
cross, but the edges $(a_{i},b_{i})$ and $(a_{i+1},b_{i+1})$ have
the same colour since $A$ and $B$ are uniformly adjacent. So this
is not possible. 

Therefore, we must have $a_{i}<b_{i}<a_{i+1}<b_{i+1}$. 

Consider for all positions $i$, we know we must have $a_{1}<b_{1}<a_{2}<b_{2}\cdots<a_{n}<b_{n}$.
So $A$ and $B$ strongly interleave.

\begin{figure}
\begin{centering}
\includegraphics{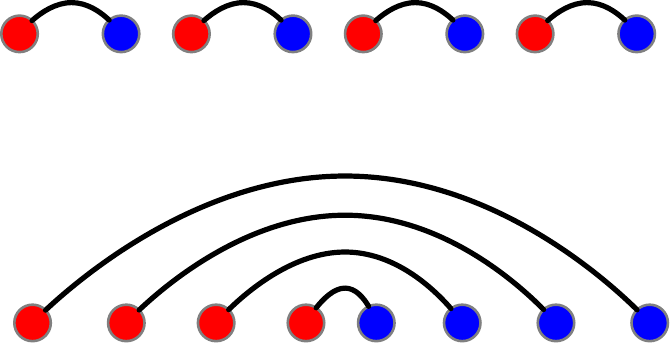}
\par\end{centering}
\caption{Bundled (top) and rainbow (bottom) sequences}
\end{figure}
\end{proof}
A pair of sequences $A,B$ satisfying the hypothesis of \propref{bundled}
is known as \textbf{bundled}.

The following observation enables us to talk about strong-interleaves
without considering two separate cases. 
\begin{fact}
\label{cor:intermediate}If $A$ and $B$ are two strongly interleaving
sequences, then between any two elements of $A$, there exist another
element of $B$.
\end{fact}

\begin{prop}
\label{prop:rainbow}Let $A,B$ be two related sequences. Suppose
$A$ and $B$ are in the opposite direction, then $A$ and $B$ must
be totally separated:
\begin{itemize}
\item If $A$ is increasing then either $a_{1}<a_{2}\cdots<a_{n}<b_{n}<b_{n-1}<\cdots<b_{1}$
or $b_{n}<b_{n-1}<\cdots<b_{1}<a_{1}<a_{2}<\cdots<a_{n}$. 
\item If $A$ is decreasing then either $a_{n}<a_{n-1}<\cdots<a_{1}<b_{1}<b_{2}<\cdots<b_{n}$
or $b_{1}<b_{2}<\cdots<b_{n}<a_{n}<a_{n-1}<\cdots<a_{1}$.
\end{itemize}
\end{prop}

\begin{proof}
Without loss of generality, $A$ is increasing. Consider whether $a_{i}<b_{i}$
, if $a_{i}<b_{i}$, then we must have $a_{n}<b_{n}$, but $a_{1}<a_{2}\cdots<a_{n}$
and $b_{n}<b_{n-1}\cdots<b_{1}$, so $a_{1}<a_{2}\cdots<a_{n}<b_{n}<b_{n-1}<\cdots<b_{1}$
follows. In the other case, $b_{1}<a_{1}$ so $b_{n}<b_{n-1}<\cdots<b_{1}<a_{1}<a_{2}<\cdots<a_{n}$.

We call a pair of sequences satisfying the assumptions in \propref{rainbow}
to be \textbf{rainbow}
\end{proof}

\subsection{Chains of Interleaves\protect 
}\begin{defn}
\label{def:k-interleave}Two sequences of points are said to $k$\textbf{-interleave},
if they are monotone in the same direction, and there is a length
$k$ subsequence $A'$ of $A$, length $k$ subsequence $B'$ of $B$
, such that $A'$ and $B'$ strongly interleave.
\end{defn}

A subsequence of $A=[a_{1},a_{2}\cdots a_{n}]$ is another sequence
$B=[b_{1},b_{2},\cdots b_{k}]$, such that there exist $k$ positions
$i_{1}<i_{2}<\cdots<i_{k}$ with $a_{j}=b_{i_{j}}$. If $A'$ is a
subsequence of $A$ and $A$ is monotone, then $A'$ is also monotone
in the same direction.
\begin{prop}
\label{prop:sub_inter} If $A,B$ strongly interleave, $A'$ is a
subsequence of length $k$, then there is a subsequence $B'$ of $B$
such that $A',B'$ strongly interleave 
\end{prop}

\begin{proof}
Regardless of how $A,B$ strongly interleave, if $A'=[a_{i_{1}},a_{i_{2}}\cdots,a_{i_{k}}]$,
then $B'=[b_{i_{1}},b_{i_{2}}\cdots b_{i_{k}}]$ strongly interleaves
with $A'$.
\end{proof}
\begin{lem}
Let $A,B,C$ be three sequences of points monotone in the same direction
with same length $n$. If $A,B$ strongly interleave, and $B,C$ strongly
interleave, then $A$ and $C$ $(\lceil\frac{n}{2}\rceil-1)$-interleave.
\end{lem}

\begin{proof}
Consider $a_{i},a_{i+1}$ and $a_{i+2}$. By \corref{intermediate},
we find some $b_{x},b_{y}$ such that $a_{i}<b_{x}<a_{i+1}$, and
$a_{i+1}<b_{y}<a_{i+2}$. Apply \corref{intermediate}, we find some
$c_{z}$ with $b_{x}<c_{z}<b_{y}$. Therefore, we have found some
$c_{z}$ with $a_{i}<c_{z}<a_{i+2}$.

For each of $i=1,3,5...$, find some $c_{t_{i}}$ with $a_{i}<c_{t_{i}}<a_{i+2}$.
This means we can find a subsequence of $c_{t_{i}}$ such that $a_{1}<c_{t_{1}}<a_{3}<c_{t_{3}}<\cdots<a_{s-2}<c_{t_{s-2}}<a_{s}$
where $s$ is the largest odd number not exceeding $n$. 

We have found $\lceil\frac{n}{2}\rceil-1$ terms of $c_{t_{i}}$ ,
so $A$ and $C$ are $(\lceil\frac{n}{2}\rceil-1)$-interleaving.
\end{proof}
By passing to subsequences using \propref{sub_inter}, we get the
following:
\begin{cor}
\label{cor:half_inter}If $A$ and $B$ $k$-interleave, $B$ and
$C$ strongly interleave, then $A$ and $C$ $(\lceil\frac{k}{2}\rceil-1)$-interleave.
\end{cor}

Since all usage of the lemma have $k$ sufficiently large to begin
with, we will write $\frac{k}{2}$ instead of $\lceil\frac{k}{2}\rceil-1$
to simplify notations. This should not affect validity of the proof.

We plan to use this lemma to a chain of strong interleaves $A_{0}\sim A_{1}\sim\cdots\sim A_{n}$.
If each of the chains is of length $k$ , then the above lemma shows
$A_{0}$ and $A_{n}$ is at least $\frac{k}{2^{n}}$ interleaving.
Consider a graph on all sequences, and two sequences are adjacent
if they strongly interleave. In this case, every two sequences in
a connected component will have large interleave. A large interleave
is a key condition to derive contradiction in the next subsection.

The above lemma shows $A_{0}$ and $A_{n}$ is at least $\frac{k}{2^{n}}$
interleaving. Even though we should not concern ourselves with the
constants, the following simple optimisation is worth mentioning. 
\begin{prop}
Let $A_{0}\sim A_{1}\sim\cdots A_{n}$ be strongly interleaving chains,
each of length $k$. Then , $A_{1}$ and $A_{n}$ are $(\lceil\frac{k}{n}\rceil-1)$-
interleaving.
\end{prop}

\begin{proof}
Let $a_{i}<a_{i+1}<\cdots<a_{i+n}$ be $n$ consecutive elements in
$A_{0}$. By \ref{cor:intermediate}, there are some elements of $A_{1}$
between each $a_{j}$ and $a_{j+1}$.

By applying this to each of $(a_{i},a_{i+1}),(a_{i+1},a_{i+2})\cdots(a_{i+n-1},a_{i+n})$,
in-between the $n$ elements of $A_{0}$ , there are at least $n-1$
elements of $A_{1}$.

Similarly, there are at least $n-2$ elements of $A_{2}$ in between
them, etc, and we finally see there is at least one element of $A_{n}$
lying between $a_{i}$ and $a_{n}$ .

This holds for every $i$, so we can find one element of $A_{n}$
in between each of $a_{0},a_{n},a_{2n},\cdots$. We can find at least
$\lceil\frac{k}{n}\rceil-1$ elements of $A_{n}$ , so $A_{0}$ and
$A_{n}$ are $\lceil\frac{k}{n}\rceil-1$-interleaving.
\end{proof}

\subsection{Crossing fans and rainbows}

We will want to derive contradictions from two crossing edges of equal
colour. 

The main way we derive contradiction will require having at least
one fan. \propref{fan-fan} gives a contradiction with two fans, \propref{fan-rainbow}
gives a contradiction with a fan and a rainbow. The previous section
showed bundled sequences preserve interleaves, \propref{rainbow_inter}
shows in addition that rainbows also preserve interleaves. The two
versions of contradictions and two versions of interleaves propagation
are combined into one lemma \lemref{ultra-inter}.
\begin{defn}
\label{def:fan}A sequence of points $A$ is said to be \emph{fanned}\textbf{
}of colour $c$ if there exist another vertex $v$ adjacent to all
$a_{i}$, and that each $v\sim a_{i}$ have colour $c$.
\end{defn}

The following notations make formal the geometric intuition of inside/outside
an edge. 
\begin{defn}
Let $e$ be an edge. Represent $e$ as $(a,b)$ where $a<b$. Then
a point $c$ is said to be \emph{inside}\textbf{ }$(a,b)$ if $a<c<b$.
It is said to be \emph{outside} $(a,b)$ if either $c<a$ or $a<b$.
\end{defn}

\begin{defn}
Two points $c,d$ are said to be on the \emph{same side }of edge $e$
, if they are both inside or both outside $e$. They are said to be
on \emph{different side }of $e$ if exactly one of $c,d$ is inside
$e$. 
\end{defn}

\begin{prop}
\label{prop:crossing}(a) Let $e$ be an edge. Edge $(a,b)$ and $e$
do not cross, if and only if $a,b$ are on the same side of $e$

(b) $a,b$ are on the same side to $(c,d)$ if and only if $c,d$
are on the same side of $(a,b)$

(c) for each $e$ , being on the same side of $e$ is an equivalence
relations on points
\end{prop}

\begin{proof}
(a) Write $e=(c,d)$. Without loss of generality, assume $a<b$ and
$c<d$ . We verify the statement by considering all 6 possible permutations
subject to these two constraints:
\begin{itemize}
\item $a<b<c<d$
\item $a<c<b<d$ (cross)
\item $a<c<d<b$
\item $c<a<b<d$
\item $c<a<d<b$ (cross)
\item $c<d<a<b$
\end{itemize}
In options $2$ and $5$, the edges cross. Options $1$ and $6$ have
$a,b$ both outside of $(c,d)$ , and options $3,4$ have $a,b$ both
inside of $(c,d)$. Thus, we have verified that $a,b$ must be on
the same side of $e$ if $(a,b)$ and $e$ do not cross. 

Conversely, since $a,b$ are on different side in options \emph{$2$
and $5$}, if $a,b$ are on the same side, then it cannot be options
$2$ or $5$, so $(a,b)$ and $e$ do not cross. 

(b) Note that crossing is a symmetric relation of two edges. By (a),
both conditions are equivalent to the condition that $(a,b)$ and
$(c,d)$ do not cross. 

(c) There are only two components, points inside $e$ and points outside
$e$. 

\end{proof}
We will frequently make passes to subsequences. We collect some basic
properties below:
\begin{prop}
\label{prop:subsequence-pass}(a) Let $A$ be fanned of colour $c$,
then any subsequence of $A$ is also fanned of colour $c$. 

(b) Let $A,B$ be rainbow of colour $c$. Then for any subsequence
$A'$ of $A$, there is some subsequence $B'$ of $B$ such that $A'$,
$B'$ is rainbow of colour $c$.

(c) If $A,B$ strongly interleave, $A'$ is a subsequence of $A$,
then there is a subsequence $B'$ of $B$ such that $A',B'$ strongly
interleave.

(d) If $A,B$ $k$-interleave, then we can find a subsequence $A'$,
$B'$ of $A,B$ respectively, such that $A'$ and $B'$ are both of
length $k$ and they strongly interleave.
\end{prop}

\begin{prop}
\label{prop:fan-fan}Let $A,B$ be sequences such that

(1) $A$ and $B$ are $k$ - interleaving for some $k\geq2$, and

(2) both $A$ and $B$ are fanned of the same colour $c$

Then some edges in the two fans crosses, leading to contradiction.
\end{prop}

\begin{proof}
We assume both $A$ and $B$ are increasing. By passing to a subsequence
we assume $A,B$ strongly interleave and have length at least $2$.
Note that being fan is preserved under taking subsequences (\propref{subsequence-pass}(a)). 

Let's assume $a_{i}$ are fanned at $a$ , $b_{i}$ are fanned at
$b$, and there are four edges of equal colour $(a_{1},a),(a_{2},a),(b_{1},b),(b_{2},b)$.
Geometrically, we can put the edges on the upper half plane. Observe
that $(a_{1},a)$ and $(a_{2},a)$ together forms a continuous curve
$\gamma_{1}$ from $a_{1}$ to $a_{2}$, and $(b_{1},b)$ and $(b_{2},b)$
forms a continuous curve $\gamma_{2}$ from $b_{1}$ to $b_{2}$.
Under the assumption that $a_{1}<b_{1}<a_{2}<b_{2}$, curves $\gamma_{1}$
and $\gamma_{2}$ must cross. 

We prove this symbolically. Let $e=(a_{1},a)$. Since $e$ have same
colour as $(b_{1},b)$, $b$, $b_{1}$ are on the same side of $e$
by \propref{crossing}. Similarly, Since $e$ have same colour as
$(b_{2},b)$ , $b,$ $b_{2}$ are on the same side of $e$. Therefore,
$b_{1}$ and $b_{2}$ are on the same side of $e$

Similarly, we can see that $b_{1}$ and $b_{2}$ are on the same side
of $(a_{2,}a)$.

Using \propref{crossing} again, $a_{1},a,a_{2}$ are on the same
side of $(b_{1},b_{2})$ , so by \propref{crossing} again $(a_{1},a_{2})$
and $(b_{1},b_{2})$ do not cross. However, regardless of the way
$A$ and $B$ 2-interleave: either $a_{1}<b_{1}<a_{2}<b_{2}$ or $b_{1}<a_{1}<b_{2}<a_{2}$,
it is easy to see that $(a_{1},a_{2})$ and $(b_{1},b_{2})$ must
cross. Contradiction.
\end{proof}
This proposition gives a contradiction with interleaving fan and rainbow. 
\begin{prop}
\label{prop:fan-rainbow}Let $A,B,C$ be three sequences and $c$
be a colour. Suppose
\begin{enumerate}
\item $A,B$ are rainbow in colour $c$,
\item A and C are $k$-interleaving for $k\geq3$, and
\item C is fanned, of colour $c$,
\end{enumerate}
then some edges in (1) and (3) crosses, leading to contradiction.
\end{prop}

\begin{proof}
We start by passing to subsequences, and assume $A,C$ strongly-interleave
with length at least $3$. 

Geometrically, rainbow sequences divide the upper half plane into
$n$ pieces, where $n$ is the length of the sequences. Fan edges
must cross these boundaries.

By reversing the enumeration index or reverse the linear order if
necessary, we may assume $a_{1}<a_{2}<a_{3}<b_{3}<b_{2}<b_{1}$. 

We abandon the old enumeration of $C$. Using \corref{intermediate},
we can find some $c_{1}\in C$ such that $a_{1}<c_{1}<a_{2}$, and
some $c_{2}\in C$ with $a_{2}<c_{2}<a_{3}$. So we have $a_{1}<c_{1}<a_{2}<c_{2}<a_{3}<b_{3}<b_{2}<b_{1}$

Let $c_{1}$ and $c_{2}$ be fanned at $c$. 

For each of $i=1,2$, Since $(a_{2},b_{2})$ and $(c_{i},c)$ do not
cross, $c$ and $c_{i}$are on the same side of $a_{2},b_{2}$ by\propref{crossing}.
Combine this for $i=1,2$ we know $c_{1}$ and $c_{2}$ are on the
same side of $a_{2},b_{2}$

This is a contradiction, since from the known order $a_{1}<c_{1}<a_{2}<c_{2}<a_{3}<b_{3}<b_{2}<b_{1}$
it is clear that $c_{1},c_{2}$ are on different side of $(a_{2},b_{2})$
\end{proof}
\corref{half_inter} showed bundles preserve interleaves. The following
proposition shows rainbows may preserve interleaves.
\begin{prop}
\label{prop:rainbow_inter}Let $A$, $B$ be two monotone sequences
of the same direction that $k$-interleave. Suppose $C$ and $D$
are monotone sequences such that $A,C$ form a rainbow, $B,D$ form
a rainbow, and the two rainbows are of the same colour, then $C,D$
$k-2$-interleave as well.
\end{prop}

\begin{proof}
By passing into appropriate subsequences of $A,B$ and $C,D$ (\propref{sub_inter}),
we may assume $A,B,C,D$ all have length $k$ and $A,B$ strongly
interleave. Without loss of generality, we assume $A,B$ are increasing,
and thus $C,D$ are assumed decreasing.

We know either $a_{1}<a_{2}<\cdots a_{k}<c_{k}<\cdots c_{1}$ or $c_{k}<\cdots c_{1}<a_{1}<\cdots a_{k}$.
In the second case, we can swap $A,C$ and swap $B,D$ , reverse the
order of $L$ and reverse enumeration order of all four sequences,
then the second case is reduced to the first case.

We abandon the old enumerations of $B$ and $D$. We use \corref{intermediate}
to find $b_{i}$ $(1\leq i\leq k-1)$ such that $a_{1}<b_{1}<a_{2}<b_{2}\cdots<a_{k-1}<b_{k-1}<a_{k}$.
We also renumber $d$ such that $(b_{1},d_{1})$, $(b_{2},d_{2})\cdots(b_{k-1},d_{k-1})$
form the rainbow.

Now we consider how $(b_{i},d_{i})$ can be positioned. Note that
$(a_{i},c_{i})$ and $(a_{i+1},c_{i+1})$ are edges same colour, and
we already know $a_{i}<b_{i}<a_{i+1}<c_{i+1}<c_{i}$. In order for
$(b_{i},d_{i})$ to not cross, we must have either $d_{i}\in[a_{i},a_{i+1}]$
or $d_{i}\in[c_{i+1},c_{i}]$. But regardless of which case it is,
we still have $a_{i}<d_{i}$. 

Recall that $D$ is decreasing, so $d_{i}<d_{i-1}$ for $2\leq i\leq k-1$.
Therefore $a_{i}<d_{i-1}$, and the option $d_{i-1}\in[a_{i-1},a_{i}]$
is not possible. We must have $d_{i-1}\in[c_{i-1},c_{i}]$. We have
thus shown that $d_{i}\in[c_{i},c_{i+1}]$ for each of $1\leq i\leq k-2$,
and so $c_{1}<d_{1}<c_{2}<d_{2}\cdots<c_{k-2}<d_{k-2}$ and $\{d_{i}\}$
and $\{c_{i}\}$ are sequences of length $k-2$ that strongly interleave.
So $C$,$D$ at least $k-2$ interleave.
\end{proof}
Finally, we unify the ideas into one lemma. 
\begin{lem}
\label{lem:ultra-inter}Let $A_{1},A_{2}\cdots A_{k}$ and $B_{1},B_{2}\cdots B_{k}$
be monotone sequences all of the same length. Suppose in addition 
\begin{enumerate}
\item For each $p$, $A_{p}$ and $B_{p}$ are monotone in the same direction,
\item For each $p$, $A_{p}$ is related to $A_{p+1}$, and $B_{p}$ is
related to $B_{p+1}$, and
\item Either 
\begin{enumerate}
\item $A_{k}$ is fanned. $B_{k}$ is fanned. Both fans are of the same
colour. Or,
\item One of $A_{k},B_{k}$ is rainbow with another sequence $C$. The other
one is fanned. The rainbow and the fan is of the same colour. 
\end{enumerate}
\end{enumerate}
Then $A_{1}$ and $B_{1}$ cannot interleave for a value larger than
$4^{k}\times10$.
\end{lem}

\begin{proof}
Note that condition 1 and 2 together implies either pairs $(A_{p},A_{p+1})$,
$(B_{p},B_{p+1})$ are both bundled, or both rainbow. 

Suppose $A_{i}$ and $B_{i}$ $k$-interleave, we claim that $A_{i+1}$
and $B_{i+1}$ at least $\frac{k}{4}$ interleave. We proceed differently
depending on $(A_{i},A_{i+1})$, $(B_{i},B_{i+1})$ are both rainbow,
or both bundled. If they are both rainbow, then using \propref{rainbow_inter}
we see that $A_{i+1}$ and $B_{i+1}$ are at least $k-2$ interleaving.
If they are instead both bundled , then by \propref{bundled}, $A_{i+1}$
strong interleave with $A_{i}$ and $B_{i+1}$ strong interleave with
$B_{i}$. By considering the chain of strong interleaves $A_{i+1}\sim A_{i}\sim B_{i}\sim B_{i+1}$
and \corref{half_inter}, we see that $A_{i+1}$ and $B_{i+1}$ are
at least $\frac{k}{4}$ interleaving. 

Suppose $A_{1}$ and $B_{1}$ interleave for $4^{k}\times10$ , then
$A_{k}$ and $B_{k}$ at least $10$ interleaves. If condition (3)(a)
is satisfied, a contradiction is reached using \propref{fan-fan}.
If condition (3)(b) is satisfied, then a contradiction is reached
using \propref{fan-rainbow} instead.
\end{proof}
In the case that all relations are bundled, we can have a more convenient
form of \lemref{ultra-inter}.
\begin{lem}
\label{lem:ultra-inter-weak}Let $A_{1},A_{2}\cdots A_{k}$ be monotone
sequences. Suppose in addition that
\end{lem}

\begin{enumerate}
\item All $A_{i}$ are monotone in the same direction 
\item $A_{i}$ bundles with $A_{i+1}$
\item Either
\begin{enumerate}
\item $A_{1}$ and $A_{k}$ are fanned of the same colour, or 
\item One of $A_{1}$,$A_{k}$ are rainbow with another sequence $C$. The
other is fanned. The rainbow and fan are of the same colour
\end{enumerate}
\end{enumerate}
Then the sequences cannot have length greater than $2^{k}\times10$
\begin{proof}
Similarly, by \propref{bundled} and \corref{half_inter} $A_{1}$
and $A_{k}$ at least $\frac{len}{2^{k}}$ interleave. If this value
is at least $3$, either \propref{fan-fan} or \propref{fan-rainbow}
gives a contradiction.
\end{proof}
In all subsequent arguments, lengths of all sequences are assumed
to be much larger than $4^{s}$ , where $s$ is the maximal length
of the interleave chains we consider, say $4(depth(T))|P|$. Therefore,
we can safely assume conditions of form ``at least $k$- interleave''
or ``length at least $k$'' are always satisfied.

\section{Notations}

\subsection{Start of Proof}

We will start by assuming $T\boxslash P$ has a valid $s$-stack layout,
then derive a contradiction from it. Let $n$ be the height of $T$,
and $m$ is the length of $P$. We will assume $n,m$ to be sufficiently
large with respect to $s$. 

We will work with $T$ being balanced rooted trees. We say $T$ is
a balanced rooted tree with \emph{degree sequence} $(d_{0},d_{1},\cdots d_{n-1})$
in the following situation: $T$ is a rooted tree, the root is of
depth $0$. Every vertex of depth $i$ has exactly $d_{i}$ children
of depth $i+1$. We will assume $n,m$ are sufficiently large compared
to $s$, and that $d_{i}$ to be sufficiently large compared to each
of $d_{i+1},d_{i+2}\cdots d_{n-1},n,m$.

\subsection{Conventions on integer arrays }

We will use integer arrays to index vertices of $T$ and $T\boxslash P$.

Integer arrays will be $1$-indexed, so $A$ will be written as $[a_{1},a_{2}\cdots a_{n}]$.
$|A|$ denotes the size of the array. This is also known as its length.
When working with arrays, capital letters denote an array, and lower
case letters denote single integers or coordinate.

If $A,B$ are two integer arrays, $A+B$ will denote the concatenation
of the two arrays in this order. That is, the $i$-th element of $A+B$
is $a_{i}$ if $i\leq|A|$, and is $b_{i-|A|}$ if $|A|<i\leq|A|+|B|$.
When $B$ consists of a single element $[v]$, we write $A+v$ instead
of $A+[v]$. We may concatenate multiple arrays and integers at one
time, such as $A+x+y+B$.

\subsection{Indexing $T$}

Vertices of $T$ will be indexed by an integer array. The root node
is given the empty array $[]$. A vertex of depth $d$ will be indexed
with an integer array of length $d$. If $v$ is given the index $A$,
then the $d_{depth(v)}$ children of $v$ will be given the indices
$A+1$, $A+2$, $\cdots,A+d_{|A|}$.

We denote by $N_{A}$ the vertex indexed by the array $A$. 

Edges of $T$ can be described as 
\[
N_{A}\sim N_{A+v}
\]

for any integer array $A$ of size between $0$ and $n$ inclusive,
and any $1\leq v\leq c_{|A|}+1$

\subsection{Indexing $T\boxslash P$}

We index vertices on $T\boxslash P$ as follows: Recall that $T\boxslash P$
consists of vertices of form $(v,i)$ where $v\in V(T),i\in[m]$.
For vertices of $T\boxslash P$, we write $(A,i)$ instead of $(N_{A},i)$.
We refer to the value $i$ as the \emph{path position}.

There are three types of edges from the product $\boxslash$. They
are:
\begin{enumerate}
\item (\emph{Vertical}) $(A+v,i)\sim(A,i)$, for any array $A$ with length
between $0$ and $n-1$, and any $v\in[d_{|A|}],i\in[m]$
\item (\emph{Horizontal}) $(A,i)\sim(A,i+1)$, for any array $A$ with length
between $0$ and $n$, and any $i\in[m-1]$
\item (\emph{Diagonal}) $(A+v,i)\sim(A,i+1)$, for any array $A$ with length
between $0$ and $n-1$, and any $v\in[d_{|A|}],i\in[m-1]$
\end{enumerate}
For the rest of the article, we will not mention the assumption $|A|\leq n$
when we write $(A,i)$. By convention, we write edges as $a\sim b$,
such that $a$ has larger depth, or $a$ has smaller path position,
or both. 

\section{Ramsey arguments on $T\boxslash P$\label{sec:ramsey-passes}}

Let $T\boxslash P$ be the given tree. Let's suppose it has a $s$-stack
layout: A linear order $L$ and a colouring $\phi:E(T\boxslash P)\rightarrow[s]$
such that equally \foreignlanguage{british}{coloured} edges do not
cross. Let the degree sequence of $T$ be $(d_{0},d_{1}\cdots d_{n-1})$.

We will make multiple ``pass to subtrees''. This means we replace
$T$ by a subtree $T'$ with a degree sequence $(d_{0}',d_{1}'\cdots d_{n}')$
such that $d_{i}'<d_{i}$.

We want to apply as much Ramsey argument as possible, making children
of each vertex to be ``as equal as possible''. Recall that the three
properties we look for in monotone, order-consistent, and uniformly
coloured in various sequences. These three assumptions can indeed
be made after several passes. 

\subsection{Pass to ensure colouring }

In the following definition, $E_{A,c}$ consists of all edges that
can be considered decedents of $N_{A+c}$: all edges that we have
control over if we delete $N_{A+c}$ from $T$. It consists of all
edges that have at least one endpoint $(B,p)$ such that $B$ have
$A+c$ as a prefix.
\begin{defn}
Let $E_{A,c}$ denotes the union of two collection of edges 

(I) Subtree edges: The sets of edges such that both endpoints are
of the form $(v,i)$ where $v$ is a descendent of $A+c$

(II) Parent edges: The sets of edges such that the endpoints are of
form $(A,i)$ or $(A+c,j)$ for some $i$.

The first kind of edges can be divided according to vertical, diagonal,
horizontal edges. For every array $X$ and integer $g,i$:
\begin{itemize}
\item (1) Vertical edges: $(A+c+X+g,i)\sim(A+c+X,i)$
\item (2) Diagonal edges: $(A+c+X+g,i)\sim(A+c+X,i+1)$
\item (3) Horizontal edges: $(A+c+X+g,i)\sim(A+c+X,i)$
\end{itemize}
The second kind of edges can be divided into two cases 
\begin{itemize}
\item (4) Parent vertical edge: $(A+c,i)\sim(A,i)$
\item (5) Parent diagonal edge: $(A+c,i)\sim(A,i+1)$
\end{itemize}
We define $P_{A}$ to be property that colour of $E_{A,c}=E_{A,c'}$
for any children $c,c'$ of $A$. This means that for example, $\phi((A+c+X+g,i)\sim(A+c+X,i))=\phi((A+c'+X+g,i)\sim(A+c'+X,i))$
for children $c,c'$, array $X$ and integer $g,i$ in the case of
vertical edges. A similar condition must hold in the other four cases.
\end{defn}

Note that the property $P_{A}$ applies to a vertex on $T$, not $T\boxslash P$.
\begin{prop}
\label{prop:pass-colour-weak}We can make a pass to a subtree $T'$
of $T$ such that property $P_{A}$ holds for all nodes $N_{A}\in T'$,
and that the degree sequence of $T'$ remains arbitrarily large.
\end{prop}

\begin{proof}
For a fixed, $A$, The set of edges $E_{A,c}$ have a bounded number
of edges, , say no more than $e=6m\times\prod_{j=i+1}^{n-1}(d_{j}+1)$.
There is at most $s^{e}$ ways to colour each $E_{A,c}$

At least $\frac{d_{i}}{s^{e}}$ children of $A$ have the subtree
$E_{A,c}$ coloured in the same way. By keeping only these children,
we have achieved property $P_{A}$ for $A$, by replacing $d_{|A|}$
with $d_{|A|}/s^{e}$.

We perform this process one layer at a time. In the $i-$th step,
we perform this for all vertex of level $n-i$ . Note that we only
need to replace $d_{n-i}$ by $d_{n-i}/s^{E}$ once. 

Since we assumed $d_{1},d_{2}\cdots,d_{n}$ are sufficiently large
to begin with , we can obtain the new $d_{1}',d_{2}',\cdots d_{n}'$
to be arbitrarily large.

After such a pass have been completed for all vertex, the property
$P_{A}$ is satisfied for all vertex $A\in V(T)$.
\end{proof}
When $P_{A}$ holds for all $A$, we get the following convenient
conclusion:
\begin{prop}
\label{prop:pass-colour}We can make a pass to a subtree $T'$ of
$T$ such that the following properties hold:
\end{prop}

\begin{enumerate}
\item Colour of $(A+v,i)\sim(A,i)$ only depends on $|A|$ and $i$. In
particular, it does not depend on values of $A$ nor $v$. In other
words, for every $i$, $(A,i)\sim(A+v,i)$ and $(B,i)\sim(B+w,i)$
are coloured the same whenever $A,B$ is of the same length.
\item Colour of $(A,i)\sim(A,i+1)$ only depends on $|A|$ and $i$.
\item Colour of $(A+v,i)\sim(A,i)$ only depends on $|A|$ and $i$.
\end{enumerate}
\begin{rem*}
In summary, the edge colours only depend on type, path position and
depth. We encode the colouring as $C(d,p,t)$ where $t\in\{\text{"Vertical","Horizontal","Diagonal"}\}$.
As a convention, we choose $d$ to be the larger depth of the two
endpoints, and $p$ to be smaller path positions of the two endpoints.
\end{rem*}
\begin{proof}
By making the pass according to \propref{pass-colour-weak}, we can
assume that the property $P_{A}$ holds for all $A\in V(T)$.

\subparagraph*{Proof of 1}

Let $A,B$ be two array of the same length, we should show that $(A,i)\sim(A+v,i)$
and $(B,i)\sim(B+w,i)$ are of the same colour.

Let's first show $(A,i)\sim(A+v,i)$ and $(B,i)\sim(B+w,i)$ have
the same colour whenever $A+v$ and $B+w$ differs in one coordinate.

If they only differ in $v$ and $w$, (i.e. $A=B)$, the claim follows
from equal colouring of edges of type (4) under $P_{A}$.

If instead $A\neq B$ and differ in one coordinate, we write $A=C+a+X$
, $B=C+b+X$. The edges belong to type (1) under $P_{C}$, so claim
follows.

For the general case of $A$ and $B$ , consider some $A=A_{0}-A_{1}-A_{2}-\cdots-A_{k}=B$
where each $A_{i}$ and $A_{i+1}$ differ in one coordinate. The above
claim shows that $(A_{j},i)\sim(A_{j}+v,i)$ have the same colour
as $(A_{j+1},i)\sim(A_{j+1}+v,i)$ . Combining this for various $j$
, we see $(A,i)\sim(A+v,i)$ and $(B,i)\sim(B+w,i)$ have the same
colour.

\subparagraph*{Proof of 2}

This case is analogous to Case I, but they fall under type (5) and
(2) respectively instead.

\subparagraph*{Proof of 3 }

Similar to case $1$, to establish the statement for general $A,B$,
it suffices to show it for whenever $A,B$ differs in exactly one
coordinate.

Write $A=C+a+X$, $B=C+b+X$ , $a\not=b$, then $(C+a+X,i)\sim(C+a+X,i+1)$
and $(C+b+X,i)\sim(C+b+X,i+1)$ have the same colour from equal colouring
of (3) under $P_{C}$.
\end{proof}
With this proposition, all colours of edges $T\boxslash P$ can be
read off from $C(d,p,t)$. We can regard the colouring information
to be exactly a $s$-colouring on the hexagonal grid of size $n\times m$.
This will be described in \ref{sec:hexgaonal}.

\subsection{Pass to ensure relative order }

The total order $L$ defines a permutation $\pi$ on all vertex of
$T\boxslash P$. 

Let $S\subset T\boxslash P$, then $L$ defines a unique sub-permutation
$\pi'$ on the $S$ vertex given by 
\[
\pi'(a)<\pi'(b)\iff\pi(a)<\pi(b)\text{ for }a\in S,b\in S
\]

\begin{defn}
For a vertex $A\in V(T)$, and a child $c$ of $A$, let $\pi_{A,c}$
to be the induced permutation of $L$ on $\{(A+c+X,i)|\text{X an integer array of any length (possibly empty)},(A+c+X)\in T,i\in[m]\}$.
In other words, $\pi_{A,c}$ is the induced permutation on all vertexes
on $T\boxslash P$ that can be considered a descendent of $N_{A+c}$,
across various path position.
\end{defn}

\begin{defn}
Let $Q_{A}$ to be the following property: 

$\pi_{A,c}=\pi_{A,c'}$ for any $c,c'$, where an element $(A+c+X,i)$
considered by $\pi_{A,c}$ is naturally corresponding to $(A+c'+X,i)$
in $\pi_{A,c'}$.

In order words, $(A+c+X,i)<(A+c+Y,j)$ iff $(A+c'+X,i)<(A+c'+Y,j)$
for any $(X,Y,i,j,c,c')$.

Note that $Q_{A}$ is a property applicable to a vertex on $T$, not
$T\boxslash P$.
\end{defn}

\begin{prop}
\label{prop:pass-order-weak}We can make a pass to a subtree $T'$
of $T$ such that $Q_{A}$ holds for all $A$. The degree sequence
of $T'$ remains arbitrarily large.
\end{prop}

\begin{proof}
For a fixed $A$, the number of vertex involved in $\pi_{A,c}$ is
bounded by, say, $v=m\times\prod_{j=i+1}^{n-1}(d_{j}+1)$. 

Therefore, there is no more than $v!$ options for $\pi_{A,c}$. At
least $\frac{d_{|A|}}{v!}$ of them have the same permutation $\pi_{A,c}.$
We only keep these $\frac{d_{|A|}}{v!}$ children. Thus by replacing
$d_{|A|}$ by $\frac{d_{|A|}}{v!}$ , property $Q_{A}$ holds.

We perform this process one layer at a time. In the $i$-th step,
we perform this process for all $v$ of depth $n-i$ simultaneously.
Afterwards, $Q_{A}$ is satisfied for any $|A|=n-1$. Note that we
only need to replace $d_{n-1}$ by $\frac{d_{n-1}}{v!}$ once. 

Note that by assuming $d_{1},d_{2}\cdots,d_{n}$ are sufficiently
large to begin with , we can obtain the new $d_{1}',d_{2}',\cdots d_{n}'$
to be arbitrarily large.

After such a pass have been completed for all vertex, the property
$Q_{A}$ is satisfied for all vertex $A\in V(T)$.
\end{proof}
When $Q_{A}$ holds for all $A$, we get the following convenient
conclusion:
\begin{prop}
\label{prop:pass-order} We can make a pass to a subtree $T'$ of
$T$ such that the following property holds:

For every four arrays $A,B,X,Y$ and two path positions $i,j$ , such
that $A$ and $B$ are of the same length, then ($X$, $Y$ may be
of different length)

$(A+X,i)<(A+Y,j)$ iff $(B+X,i)<(B+Y,j)$.
\end{prop}

\begin{rem*}
We will apply this lemma by specifying a 6-tuple $(A,B,X,Y,i,j)$.
\end{rem*}
\begin{proof}
Directly from $Q_{A}$, we obtain that for any $(X,Y,i,j,c,d)$ :
$(A+c+X,i)<(A+c+Y,j)$ iff $(A+d+X,i)<(A+d+Y,j)$. 

Now it remains to conclude the ordering in the described form. 

Given $A,B,X,Y,i,j$ with $|A|=|B|$. Suppose $A$ and $B$ differs
in one coordinate only. Write $A=I+c+J$, $B=I+d+J$. Apply the above
claim using $Q_{I}$ and the 6-tuple $(J+X,J+Y,i,j,c,d)$. We obtain
$(I+(c+J+X),i)<(I+(c+J+Y),j)$ iff $(I+(d+J+X),i)<(I+(d+J+Y),j)$.
Thus, the proposition have been established in the case that $A$
and $B$ differs in exactly one coordinate.

For general $A,B$ , we consider some chains of array $A=A_{0}-A_{1}-A_{2}\cdots-A_{k}=B$,
such that $A_{i}$ and $A_{i+1}$ differs in exactly one coordinate.
By applying the preceding paragraph to each of $(A_{i},A_{i+1})$,
we see $(A_{i}+X,i)<(A_{i}+Y,j)$ iff $(A_{i+1}+X,i)<(A_{i+1}+X,i)$.
Together, this implies $(A+X,i)<(A+Y,j)$ iff $(B+X,i)<(B+Y,j)$.
\end{proof}
In the rest of the proof, we will use this property to show they are
various sequences are order-consistent. 

\subsection{Pass to ensure monotone and lexicographical ordering}

We use $*$ to notate a sequence of vertices of $T\boxslash P$ as
follows. Notation of a sequence consists of exactly one $*$. The
$*$ means this value is the only difference between the element of
the sequence, and is understood to be taken from $1$ to some maximum,
either $d_{i}$ or $m$. For example $([1,*],3)$ refers to the sequence
$([1,1],3),([1,2],3),([1,3],3)\cdots$.

By using a Ramsey argument via a generalization of Erdos-Szekeres
theorem for multidimensional array, we can assume a strong condition
that the vertex of the same layer are lexicographically ordered. We
follow the notations in \cite{erdosgoodbound}.
\begin{defn*}
(Lex-monotone array) \label{def:lex-monotone}A $d$-dimensional array
$f$ is said to be \emph{lex-monotone} if there exist a permutation
$\sigma:[d]\rightarrow[d]$ and a sign vector $s\in\{INC,DEC\}^{d}$
such that

Let $i$ be the smallest integer such that position $\sigma(i)$ of
${\bf x}$ and ${\bf y}$ differ, then 

if $s(i)=INC$ , then $f({\bf x})<f({\bf y})$ iff ${\bf x}_{\sigma(i)}<{\bf y}_{\sigma(i)}$

if $s(i)=DEC$, then $f({\bf x})>f({\bf y})$ iff ${\bf x}_{\sigma(i)}<{\bf y}_{\sigma(i)}$
\end{defn*}
\begin{fact}
\label{fact:normal-monotone}If $f$ is a lex-montone array, then
the sequence $f(a_{1},a_{2},\cdots,a_{k-1},*,a_{k+1},\cdots a_{d})$
is monotone. Moreover, the direction of monotone only depends on the
position of $*$ , and does not depends on values of $a_{i}$.
\end{fact}

\begin{proof}
The direction of $f(a_{1}\cdots a_{k-1},*,a_{k+1}\cdots,a_{d})$ is
exactly $s(k)$.
\end{proof}
The following result can be found in \cite{erdosgoodbound}.
\begin{thm}[Erdos-Szekeres theorem for lex-monotone]
For any $d$ and $n$, there exist an integer $L_{d}(n)$ such that
for any $N\geq L_{d}(n)$, any $N\times N\times....\times N$, $d$-dimensional
integer array have a $n\times n\times\cdots n$ lex-montone subarray.
\end{thm}

\begin{prop}
\label{prop:pass-lex-weak}We can pass to a subtree, such that all
vertex of the same level and same path position forms a lex-montone
array.

More formally, for fixed length $len$ and $p\in[m]$, the $len$-dimensional
array 

\[
L([*,*,*,....,*],p)
\]
forms a lex-montone array. The degree sequence of the subtree can
remain arbitrarily large.
\end{prop}

\begin{proof}
For each level $i$ from $0$ to $n-1$ , suppose we want the new
$T$ to have degree sequence $d_{0}',d_{1}'\cdots d_{n-1}'$ then
as long as $d_{0},d_{1}\cdots d_{i}$ are all greater than $L_{i+1}(\max_{j}d_{j}')$,
then we can pass it into a subset of the children, such that $L([*,*,*,....,*],p)$
is lex-montone. 

After performing this procedure for $m\times n$ times, once for each
level and path position , the desired lex-montone property holds for
every level and path position.
\end{proof}
\begin{cor}
\label{cor:pass-lex-basic}After passing to subtree as in \propref{pass-lex-weak},
any sequences of form $(A+*+B,p)$ are monotone.
\end{cor}

\begin{proof}
Follows from \factref{normal-monotone}.
\end{proof}
As for the direction of $(A+*+B,p)$, we shall encode them in the
function $Z(i,j,p)$. 
\begin{defn}
\label{def:Z}The function $Z(i,j,p)$ takes value in $\{INC,DEC\}$.
This is defined for $1\leq i\leq j\leq n$ and $p\in[1,m]$. This
is the direction of any/all (monotone) sequences $(A+*+B,p)$ where
$*$ is at position $i$ and the total length $|A|+1+|B|$ is equal
to $j$.
\end{defn}

The values $Z(*,len,p)$ is obtained as the sign vector $s$ obtained
from the lex-monotone $L([*,*,*,....,*],p)$ with $len$ '{*}'s.

By combining the extra assumption we get from all three passes \propref{pass-colour},\propref{pass-order},\corref{pass-lex-basic},
we can find many related sequences.
\begin{prop}
\label{prop:main-related}

(1) Sequences $(A+*+Y+v,p)$ and $(A+*+Y,p)$ are related for any
arrays $A,Y$ and $p\in[m]$. The relation colour only depends on
$(|A|,|Y|,p,v)$ and not values of $A$ or $Y$.

(2) Sequences $(A+*+Y+v,p)$ and $(A+*+Y,p+1)$ are related for any
arrays $A,Y$ and $p\in[m-1]$. The relation colour only depends on
$(|A|,|Y|,p,v)$.

(3) Sequences $(A+*+Y,p)$ and $(A+*+Y,p)$ are related for any arrays
$A,Y$ and $p\in[m]$. The relation colour only depends on $(|A|,|Y|,p)$.
\end{prop}

\begin{proof}
For (1), to show $(A+*+Y+v,p)$ and $(A+*+Y,p)$ are related, we have
to check monotone, order consistent and uniformly adjacent.

\subparagraph{Monotone}

They are both monotone by \corref{pass-lex-basic}.

\subparagraph{Order Consistent}

Apply \propref{pass-order} to the 6-tuple $(A+c,A+c',Y+v,Y,p,p)$
we get $(A+c+Y+v,p)<(A+c+Y,p)$ iff $(A+c'+Y+v,p)<(A+c'+Y,p)$, meaning
$(A+*+Y+v,p)$ and $(A+*+Y,p)$ are order consistent.

\subparagraph*{Uniformly Adjacent}

There is an vertical edge between $(A+a+Y+v,p)$ and $(A+a+Y,p)$
for each $a$. All of them have colour $C(|A|+|Y|+2,p,\text{vertical})$
as given by \propref{pass-colour}.

For (2), the proof is analogous. But we use the 6-tuple $(A+c,A+c',Y+v,Y,p,p+1)$
instead. And they have colour $C(|A|+|Y|+2,p,\text{diagonal})$ instead.

For (3), the proof is very similar. 

\subparagraph{Order consistent}

Apply \propref{pass-order} to the $6$-tuple $(A+c,A+c',Y,Y,p,p+1)$,
we get $(A+c+Y,p)<(A+c+Y,p+1)$ iff $(A+c'+Y,p)<(A+c'+Y,p+1)$. This
shows $(A+*+Y,p)$ and $(A+*+Y,p)$ are order consistent.

\subparagraph{Colouring}

All of the edges have the colour $C(|A|+|Y|+1,p,\text{horizontal})$
\end{proof}
In fact, the permutation of lexicographical ordering must go in the
order of tree depth. In order to show this, we must invoke \lemref{ultra-inter}.
\begin{lem}
\label{lem:identity-permutation}There is a constant $C=C(n)$ such
that if each $d_{i}\geq C(n)$, The $\sigma$ involved in each lex-montone
is exactly the identity permutation. That is, the lex monotone ordering
comparing $(A,p)$, $(B,p)$ will first consider the first coordinate
of $A,B$, and if they are equal, consider the second coordinate,
and so on.
\end{lem}

\begin{proof}
Fix path position $p$and length $len$. Suppose $\sigma$ has any
inversion, such as it reading position $i$ before position $j$,
where $i>j$.

We use $I_{i}$ to denote an array of length $i$ that consists of
entirely '$1$'s.

Consider the two sequences $(A+1+B+*+I_{t},p)$ and $(A+2+B+*+I_{t},p)$,
where $A,B$ are arbitrary (such as consisting of entirely '1's),
the value $1$/ value $2$ is in position $i$, the $*$ is in position
$j$, the whole array is of size $len$. Therefore, $t=len-j$.

Elements of these two sequences only possibly differ in position $i,j$
, and since we compare elements by considering position $i$ first
then position $j$, it is easy to see that $(A+1+B+*+I_{t},p)$ and
$(A+2+B+*+I_{t},p)$ are strongly interleaving by definition.

Let $S_{i}=(A+1+B+*+I_{t-i},p)$ , $Q_{i}=(A+2+B+*+I_{t-i},p)$. We
claim that the chains $S_{0}\sim S_{1}\cdots\sim S_{t}$ and $Q_{0}\sim Q_{1}\cdots\sim Q_{t}$
satisfy the hypothesis in \lemref{ultra-inter}. We verify the conditions
(1), (2) (3) in \lemref{ultra-inter}.

Condition (1): We see $S_{i}=(A+1+B+*+I_{t-i},p)$ and $Q_{i}=(A+2+B+*+I_{t-i},p)$
are monotone in the same direction $Z(len-i,j,p)$.

Condition (2): By \propref{main-related}, each $S_{i}$ and $S_{i+1}$
are related, each $Q_{i}$ and $Q_{i+1}$ are related.

Condition (3): $(A+1+B+*,p)$ is fanned to $(A+1+B,p)$ and $(A+2+B+*,p)$
is fanned to $(A+2+B,p)$. The colour of the two fans are the same:
it is $C(|A|+|B|+2,\text{p},\text{Vertical})$.

So, all conditions for the \lemref{ultra-inter} have been satisfied.
It follows from \lemref{ultra-inter} that $(A+1+B+*+X,p)$ and $(A+2+B+*+X,p)$
cannot $10\times4^{|X|}$ interleave. So by taking $C=10\times4^{n}$
the above condition must fail. In this case, $\sigma$must be the
identity permutation.
\end{proof}
As a consequence of \lemref{identity-permutation}, $Z(i,j,p)$ enables
to compare any $(A,p)$,$(B,p)$ with $|A|=|B|$. The following is
stated for clarity.
\begin{fact}
\label{prop:pass-lex-2}Let $A,B$ be two different integer arrays
of the same length, let $p\in[m]$ be an integer. Then $(A,p)$ and
$(B,p)$ can be compared as follows: 

Let $i$ be the first position in which $A$ and $B$ differs,

(1) If $Z(i,|A|,p)$ is INC, then$(A,p)<(B,p)$ iff $A_{i}<B_{i}$

(2) If $Z(i,|A|,p)$ is DEC, then $(A,p)<(B,p)$ iff $A_{i}>B_{i}$ 
\end{fact}

\subsection{Property of $Z(i,j,p)$}

Not all combinations of $Z(i,j,p)$ are possible. $Z(i,j,p)$ have
a property that is preserved one way as we move across $i$. 
\begin{lem}
\label{lem:two-d}Let $A,B$ be $2$-dimensional lex-monotone arrays
with length at least $3$ in each coordinate, with permutations $\sigma_{A},\sigma_{B}$
and sign vector $s_{A},s_{B}$ respectively. Suppose both $\sigma_{A},\sigma_{B}$
are both the identity permutation, and that $A(i,*)$ are related
to $B(i,*)$ for each $i$. Then we cannot have $s_{A}(1)\neq s_{B}(1)$
and $s_{A}(2)=s_{B}(2)$ simultaneously.
\end{lem}

\begin{proof}
Without loss of generality, assume $s_{A}(1)=s_{A}(2)=increasing$,
we can reduce to this case by possibly reversing enumeration of both
$A,B$, on the first or second dimension.

Suppose we have $s_{A}(2)=s_{B}(2)$. From $s_{A}(2)=s_{B}(2)$ and
the assumption that $A_{i,*}$ is related to $B_{i,*}$, we know $A_{i,*}$
strongly interleaves with $B_{i,*}$ by \propref{bundled}, and thus
$B_{i,1}<A_{i,2}<B_{i,3}$ for each $i$. 

By assumption $s_{A}(1)=increasing$, so $A_{1,2}<A_{2,2}$. Therefore
$B_{1,1}<A_{1,2}<A_{2,2}<B_{2,3}$. Thus, $B_{1,1}<B_{2,3}$. Since
$(1,1)$ and $(2,3)$ have their first coordinate differs, by definition
of lex-monotone $B_{1,1},B_{2,3}$ are compared using $s_{B}(1)$.
Therefore, we know $s_{B}(1)=increasing$. So $s_{A}(1)=s_{B}(1)$.
\end{proof}
It is not difficult to see that the other three combinations: ($s_{A}(1)=s_{B}(1)$,
$s_{A}(2)=s_{B}(2)$), ($s_{A}(1)=s_{B}(1)$, $s_{A}(2)\neq s_{B}(2)$),
$(s_{A}(1)\neq s_{B}(1),s_{A}(2)\neq s_{B}(2))$ are all possible.
\begin{lem}
\label{lem:z-consistent}(a) For any $1\leq p\leq m$, $k\geq2$ and
$k\leq i$ , and $Z(k,i+1,p)=Z(k,i,p)$, then $Z(k-1,i+1,p)=Z(k-1,i,p)$.

(b) For any $1\leq p\leq m$, $k\geq2$ and $k\leq i$, and $Z(k,i+1,p)=Z(k,i,p+1)$,
then $Z(k-1,i+1,p)=Z(k-1,i,p+1)$.

(c) For any $1\leq p\leq m$, $k\geq2$ and $k\leq i$, and $Z(k,i,p)=Z(k,i,p+1)$,
then $Z(k-1,i,p)=Z(k-1,i,p+1)$.
\end{lem}

\begin{rem*}
The conditions $k\leq i$ and $k\geq2$ have no special meanings,
other than the fact that all $Z(a,b,p)$ involved are defined. Recall
that $Z(a,b,p)$ is only defined for $a\geq1$ and $a\leq b$. The
three statements are corresponds to vertical, diagonal and horizontal
sequences.
\end{rem*}
\begin{proof}
(a): Suppose $Z(k,i+1,p)=Z(k,i,p)$. Consider the two dimensional
arrays $T_{1}(x,y)=(A+x+y+X+1,p)$ and $T_{2}(x,y)=(A+x+y+X,p)$.
The arrays $A$ and $X$ consists of arbitrary integers, but its length
are chosen such $x$ is at position $k-1$, $y$ is at position $k$,
and $T_{1}$ consists of elements of depth $i+1$, and $T_{2}$ consists
of elements of depth $i$. By \propref{pass-lex-weak},\lemref{identity-permutation}
$T_{1}$ and $T_{2}$ are indeed lex-monotone with the identity permutation.
$T_{1}(x,*)=(A+x+*+1,p)$ is related to $T_{2}(x,*)=(A+x+*,p)$ by
\ref{prop:main-related}.If $s_{1}$ and $s_{2}$ are the sign vectors,
the assumption $Z(k,i+1,p)=Z(k,i,p)$ says $s_{1}(2)=s_{2}(2)$. The
hypothesis of \lemref{two-d} is satisfied, and \lemref{two-d} shows
we must have $s_{1}(1)=s_{2}(1)$, which means $Z(k-1,i+1,p)=Z(k-1,i,p)$

(b) The proof is exactly the same, but we use $T_{1}(x,y)=(A+x+y+X+1,p)$
and $T_{2}(x,y)=(A+x+y+X,p+1)$ instead.

(c) The proof is exactly the same, but we use $T_{1}(x,y)=(A+x+y+X,p)$
and $T_{2}(x,y)=(A+x+y+X,p+1)$ instead, and such that $T_{1},T_{2}$
both consists of elements of depth $i$. 
\end{proof}

\section{The Hexagonal Grid \label{sec:hexgaonal}}

In the rest of the proof, we attempt to find a contradiction among
$Z(i,j,p)$, using \lemref{z-consistent} and \lemref{ultra-inter-weak}.
We shall need properties about hexagonal grids. Firstly, we shall
regard $Z(i,j,p)$ as the $2-$colouring of many hexagonal grids.
\begin{defn}
(Hexagonal Grid) The $n\times m$ hexagonal grid consists of vertex
$(i,j)$ for $i\in[n],j\in[m]$. It consists of three types of edges:
horizontal edges, vertical edges, and diagonal edges, in a fashion
similar to $T\boxslash P$: It consists of edges for each $i,j$:
\begin{itemize}
\item $(i+1,j)\sim(i,j)$
\item $(i+1,j)\sim(i,j+1)$
\item $(i,j)\sim(i,j+1)$
\end{itemize}
\end{defn}

Th hexagonal grid will be frequently denoted by $H$. We refer to
edges in a hexagonal grid by ``$H$-edges'' , in order to not confuse
with edges in $T$ or $T\boxslash P$.

We can consider $Z(i,j,p)$ as a $2$-colouring of $n$ different
Hexagonal grids. The $i$-th board consists of vertex with coordinates
$(x,y)$ where $x\geq i$, and its colouring information are taken
from $Z(i,*,*)$. In this case, the information from \lemref{z-consistent}
correspond to one way preservation as we move between grids.
\begin{cor}
\label{cor:chromatic-perserve}(1) If $C$ is a monochromatic component
in $Z(i,*,*)$, then $C$ is also monochromatic in $Z(i-1,*,*)$ (not
necessarily a maximal monochromatic component). 

(2) If $a$ and $b$ are two adjacent points in $Z(i,*,*)$ and they
are of different colour, then $a,b$ is of different colour in $Z(i+1,*,*)$
as well.
\end{cor}

\begin{proof}
Combine the three parts of \lemref{z-consistent}. (2) is the contrapositive
of (1).
\end{proof}
The hexagonal board is interesting when being $2$- coloured. We can
see this using the dual graph.
\begin{defn}
(Hexagonal Dual graph) The dual graph of the hexagonal grid is a section
of the tiling of the two dimensional plane by regular hexagons. Each
vertex corresponds to a hexagon. Three hexagons meet at one point.
A non-border vertex of the dual graph have degree $3$.

There are two situations where three hexagons meet:
\begin{itemize}
\item $(i,j),(i-1,j),(i-1,j+1)$
\item $(i,j),(i,j+1),(i-1,j+1)$
\end{itemize}
The dual graph $\bar{H}$ have vertices being points where three hexagon
meets, and edges being edges of the hexagon.
\end{defn}

\begin{defn}
\label{def:dual-graph-label}(Dual graph Indexing) We index the dual
graph as follows 
\begin{itemize}
\item $(i,j,-)$ denotes the intersection point of hexagons $(i,j),(i-1,j),(i-1,j+1)$
\item $(i,j,+)$ denotes the intersection points of hexagons $(i,j),(i,j+1),(i-1,j+1)$
\end{itemize}
The indexing once again follows the convention that the higher depth
and smaller path position is taken.

The three types of edges are 
\begin{itemize}
\item $(i,j,-)\sim(i,j,+)$
\item $(i,j,-)\sim(i-1,j,+)$
\item $(i,j,-)\sim(i,j-1,+)$
\end{itemize}
\begin{figure}
\begin{centering}
\includegraphics{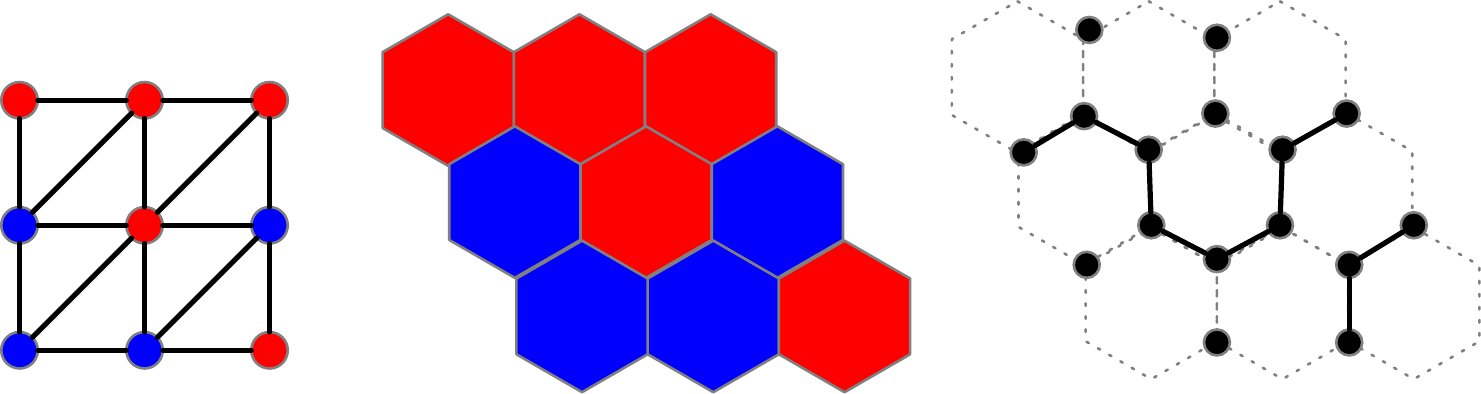}
\par\end{centering}
\caption{2-colouring of a hexagonal grid of size $3$, its dual graph, and
the colour boundary graph $\bar{H}'$. Vertices on $\bar{H}'$ involving
in only one hexagon are removed.}

\end{figure}
\end{defn}

Suppose vertex of the hexagonal grid is being $2$-coloured (say,
red-blue colored), i.e. each hexagon on the dual graph is given one
of two colours, let $\bar{H}'$ be the subgraph of $\bar{H}$ that
consists of edges which border two hexagons of different colour. For
every generic vertex $v$of $\bar{H}$, either $0,1,2,3$ hexagons
bordering $v$ are being red coloured, then $d_{\bar{H}'}(v)=0,2,2,\text{or},0$
respectively. However, since in all applications we will only consider
a finite piece of the hexagonal board, the situation is different
on the borders. If $v$ only borders $2$ hexagons, then $d_{\bar{H}^{'}}(v)=0,1$,
and $1$ only when the two hexagons differ. Since $\bar{H'}$ have
every degree being $0,1,2$, it is easy to see that the edges decompose
into paths and cycles. Moreover, every degree $1$ vertex must be
start of some path, that ends at another degree $1$ vertex, which
is only possible to be at a border. This property is the key observation
behind the hex lemma:
\begin{lem}
\label{lem:hex-lemma}(Hex Lemma) Given a two colouring of the $n\times n$
hexagonal grid, there must exist a monochromatic path of length at
least $n$.
\end{lem}

This can be proved by padding $n\times n$ grid with a border of certain
colour. The details can be found on \cite{hex}. The Hex lemma itself
is used in the first proof of unbounded stack number with queue number
4 \cite{key-first-4}. 

For our result, we need more than monochromatic paths in $Z(i,j,p)$
to derive a contradiction. The hex lemma itself is not useful for
our purposes, as mentioned in \subsecref{comparison}. The only way
we can derive a contradiction is with at least a fan (recall \lemref{ultra-inter}).
Only vertices corresponding to $Z(i,i,p)$ can be made into a fan,
whereas others $Z(i,j,p)$ for $i<j$ can only be made at best a rainbow.
This is the reason for many technicalities in the remainder of the
proof, but it turns out the property is strong enough to arrive at
a contradiction. 

The following lemma shows the main way information travels across
various $Z(i,*,*)$. The chromatic boundary line here is exactly a
path or cycle on $\bar{H'}$, i.e. a continuous boundary separating
the two colours of the hexagons. However, we choose to label it in
the original graph $H$ instead. 
\begin{defn}
\label{def:chromatic-boundary}Consider a hexagonal grid $H$ and
a 2-colouring of $H$. Let $A,B$ be sequences of points in $H$.
Write $A=(a_{1},a_{2}\cdots a_{n})$, $B=(b_{1},b_{2}\cdots,b_{n})$,
where each $a_{i}$ and $b_{i}$ are of form $(x,p)$ for $1\leq x\leq n$
and $1\leq p\leq m$. We say parts $A$ and $B$ form a chromatic
boundary line if the following are satisfied: 
\begin{enumerate}
\item All $a_{i}$ have the same colour. All $b_{i}$ have the other colour. 
\item Each $(a_{i},b_{i})$ is an $H$-edge
\item For each $i\in\{1,2,\cdots n-1\}$, exactly one of the following holds,
$a_{i}=a_{i+1}$, $b_{i}=b_{i+1}$
\item All pairs $(a_{i},b_{i})$ are distinct
\end{enumerate}
\end{defn}

\begin{lem}
\label{lem:boundary-preserve}Let $(A,B)$ be a chromatic boundary
line on $Z(i,*,*)$. Provided that all points of $(A,B)$ are present
on $Z(i+1,*,*)$, then $(A,B)$ is a chromatic boundary line on $Z(i+1,*,*)$
as well.
\end{lem}

\begin{proof}
Consider $a_{i},b_{i},a_{i+1},b_{i+1}$. Without loss of generality,
we assume $a_{i}=a_{i+1}$. As $a_{i}$ is coloured differently from
$b_{i}$, and differently form $b_{i+1}$, these two facts are preserved
as we move to $Z(i+1,*,*)$ by \corref{chromatic-perserve}. In particular,
$b_{i}$ have the same colour as $b_{i+1}$. Therefore, property $(1)$
holds. Properties $(2),(3),(4)$ does not take into account which
grid $Z(i,*,*)$ it is in, so they automatically holds. 
\end{proof}
\begin{defn}
Recall that from \propref{pass-colour}, colour of all $T\boxslash P$
edges $(A,i)\sim(B,j)$ only depends on the depth and path position.
We can define the colour of an $H$-edge $(x,p)\sim(y,q)$ to naturally
be the colour of all/any $(A,p)\sim(B,q)$ where $|A|=x$ and $|B|=q$. 
\end{defn}

We transform \lemref{ultra-inter} into a condition that derive contradiction
from certain types of chromatic boundary line.
\begin{lem}
\label{lem:k-points-sep}There cannot exist a chromatic boundary and
one of the $s$ colours $c$ with the following property: There are
two positions on the chromatic boundary $(a_{i},b_{i})$ , $(a_{j},b_{j})$,
such that 
\begin{enumerate}
\item the colour of $H$-edges $(a_{i},b_{i}),(a_{j},b_{j})$ are both $c$. 
\item $(a_{i},b_{i})$ is a vertical $H$-edge and $a_{i}$ have greater
depth. The same holds for $(a_{j},b_{j})$
\item When transversing from $a_{i}$ to $a_{j}$ along the chromatic boundary,
all intermediate $a_{t}$ have $depth(a_{t})\geq min(depth(a_{i}),depth(a_{j}))$
\end{enumerate}
\end{lem}

\begin{proof}
Let's consider only the $A$-side of the chromatic boundary. Write
this as $a_{i}\sim a_{i+1}\sim\cdots\sim a_{j-1}\sim a_{j}$. Write
each $a_{k}$ as $a_{k}=(d_{k},p_{k})$. Let $s=min(depth(a_{i}),depth(a_{j}))$.
Because chromatic boundaries are preserved as we move from $Z(a,*,*)$
to $Z(b,*,*)$ for $a<b$, we can assume this chromatic boundary is
situated in $Z(s,*,*)$.

For each vertex $a_{k}=(d_{k},p_{k})$, we fix a corresponding sequence
$S_{k}=(A+*+X_{k},p_{k})$ in $T\boxslash P$. The arrays $A$ and
$X_{k}$ are chosen such that $*$ is at position $s$ and the total
length is of $d_{k}$. The actual values of $A$ and $X_{i}$ are
arbitrary as long as they are consistent across various $i$, so we
will just assume they consists of entirely '1's. Note that the assumption
$depth(a_{t})\geq min(depth(a_{i}),depth(a_{j}))$ means that $|X|\ge0$
so the sequence exist for all intermediate $a_{t}$ between $a_{i}$
and $a_{j}$. Note that $S_{k}$ is monotone in the direction of $Z(s,d_{i},p_{i})$.
By definition of the chromatic boundary, we know $Z(s,d_{i},p_{i})=Z(s,d_{i+1},p_{i+1})$.
This means $S_{i}$ is monotone in the same direction with $S_{i+1}$.
Therefore, the chain $S_{i}\sim S_{i+1}\sim\cdots S_{j}$ is a chain
of strong interleaves by \propref{bundled}.

By the second assumption, $b_{i}=(d_{k}-1,p_{k})$. If $depth(a_{i})=s$,
then $S_{i}$ is fanned of colour $c$: $S_{i}=(A+*,p_{i})$ is fanned
to $(A,p_{i})$. By assumption (1), this fan is of colour $c$. If
instead $depth(a_{i})>s$, then $(A+*+X_{i},p_{i})$ is in rainbow
with $(A+*+Y,p_{i})$ of colour $c$ by assumption (1), where $Y$
is $X_{i}$ removing the last element. 

As $s=min(depth(a_{i}),depth(a_{j}))$ , so at least one of $depth(a_{i})=s$
, $depth(a_{j})=s$ holds, so we have at least one fan , and the other
one is in a fan or rainbow. In both cases, all assumptions in \lemref{ultra-inter-weak}
are satisfied. Therefore, this leads to a contradiction by \lemref{ultra-inter-weak},
since all sequences can be taken to be arbitrarily long. 
\end{proof}
\begin{cor}
\label{cor:k-points}There cannot exist a chromatic boundary line
and $s+1$ positions $(a_{t_{1}},b_{t_{1}})$,$(a_{t_{2}},b_{t_{2}})\cdots$,
$(a_{t_{s+1}},b_{t_{s+1}})$ such that 
\begin{enumerate}
\item For each $i$, $(a_{t_{i}},b_{t_{i}})$ are vertical $H$-edge, and
$a_{t_{i}}$ have greater depth.
\item For any $i,j$, when transversing from $a_{t_{i}}$ to $a_{t_{j}}$
along the chromatic boundary, all intermediate $a_{t}$ have $depth(a_{t})\geq min(depth(a_{t_{i}},a_{t_{j}}))$
\end{enumerate}
\end{cor}

\begin{proof}
Note that every pairs of positions $(a_{t_{i}},b_{t_{i}})$ and $(a_{t_{j}},b_{t_{j}})$
satisfy the conditions from the previous lemma. Since there are $s+1$
positions and $s$ different colours, at least one pair must have
the same colour as an $H$-edge. This contradicts the previous lemma. 
\end{proof}
The points on the first layer $Z(1,1,*)$ do not fall under our definition
of chromatic boundary. So we have to write it separately.
\begin{cor}
\label{cor:k-points-top}We cannot have $s+1$ different values $k_{1},k_{2}\cdots k_{s+1}$
such that $Z(1,1,k_{i})$ all belong to the same monochromatic component
in $Z(1,*,*)$.
\end{cor}

\begin{proof}
Let $S_{i}=(*,k_{i})$, then each of them is fanned to $([],k_{i})$
of some colour. Pick two of them $k_{i}$ and $k_{j}$ such that the
fans have the same colour. Then, there is some $Z(1,1,k_{i})=Z(1,d_{1},p_{1})=Z(1,d_{2},p_{2})\cdots=Z(1,1,k_{j})$,
where each of $(d_{i},p_{i})$ is adjacent to $(d_{i+1},p_{i+1})$,
since they belong to the same monochromatic component. In a similar
way with \lemref{k-points-sep}, we can transform each of these into
a chains of interleaving sequences: for $(d_{i},p_{i})$ we consider
$(*+X,p_{i})$ where $X$ is of length $d_{i}-1$. So we have found
some chain of strong interleaves from $S_{i}$ to $S_{j}$ , this
satisfied the assumptions of \lemref{ultra-inter-weak} so leads to
a contradiction.
\end{proof}

\section{Proof of main theorem \label{sec:main-proof}}

Finally, we are ready to give proof of the main result. We will show
that on a sufficiently large hexagonal grid, we can always either
find a chromatic boundary that satisfy the assumption on \corref{k-points},
or violate \corref{k-points-top}, leading to a contradiction.
\begin{lem}
\label{lem:long-line-bad}There exist a constant $S=S(s)$ such that
if $L$ is a chromatic boundary of length greater than $S$, then
we can find $s+1$ points satisfying the assumptions of \corref{k-points}.
\end{lem}

\begin{proof}
We will instead work in the dual graph to find the $s+1$ points.
Recall that the dual graph are given labels $(i,j,+)$ or $(i,j,-)$
(\defref{dual-graph-label}). We will try to consider the top-most
points of the boundary $L$ interpreted as a continuous curve on $\bar{H}'$.
Symbolically, we take the dual graph vertex of $L$ with the lowest
depth, meaning we consider the first vertices in the order $(1,*,-)<(1,*,+)<(2,*,-)<(2,*,+)<\cdots$
that are present in $L$. Both $(i,j,+)$ and $(i,j,-)$ are said
to have depth $i$, but we priorities $(i,j,-)$ over $(i,j,+)$ when
picking points of lowest depth. If a point of lowest depth is $(i,j,+)$,
notice that two of its three neighbors $(i,j,-),(i,j+1,-)$ have even
lower depth, so this can only happen at the two endpoints of $L$.
By shrinking $L$ on both ends by upto $1$ , we may \foreignlanguage{british}{assume}
the point with lowest depth are of form $(i,j,-)$. Let $i_{0}$ be
the lowest depth of $L$, we say all points on $L$ of form $(i_{0},j,-)$
are $L$\emph{-critical} points. Clearly, any boundary $L$ admits
at least one critical point. 

For a point $(i,j,-)$, exactly two \foreignlanguage{british}{neighbours}
of $(i,j,-)$ have higher depth: $(i,j,+)$ and $(i,j-1,+)$. If a
critical point occur does not occur at endpoints of $L$, then the
two edge connected to it must be $(i,j,+)$ and $(i,j-1,+)$ in some
order. This in turn, means that the point $(i,j,-)$ is the meeting
of the three hexagons $(i,j),(i-1,j),(i-1,j+1)$, such that $(i-1,j)$
and $(i-1,j+1)$ have the same colour, and $(i,j)$ is coloured differently
from them. In the proof of this lemma we denote the two colours as
red, blue instead of $\{INC,DEC\}$. A critical point $(i,j,-)$ is
\emph{red-based}\textbf{,} if $(i,j)$ have the colour red. It is
\emph{blue-based} if $(i,j)$ have colour blue. 

Let's call a pair of critical points $X,Y$ (not necessarily with
respect to the same chromatic boundary $L$) \emph{good}, if they
are of different base colour, or
\begin{itemize}
\item If $X$ is critical with respect to some $L$, $Y$ belongs to $L$,
all critical points between $X$ and $Y$ along $L$ have the same
base as $X$
\item The same condition with $X,Y$ swapped.
\end{itemize}
Suppose we picked out $2s+1$ pairwise good points. Then, we pick
some $s+1$ of them with the same colour base, say blue based. For
each critical point $(i,j,+)$ we will pick the $H-$edge $(a,b)$
where $a=(i,j)$, $b=(i-1,j)$ for \corref{k-points}. We claim that
pairs being good imply the second assumption of \corref{k-points}
is satisfied.

Let $X,Y$ be two chosen critical points. Assume the first condition
of being good holds: that $X$ is critical to some $L$ where $Y$
belongs to $L$, and that all critical points between $X$ and $Y$
are blue based. Let $X=(i,j)$ , $Y=(i',j')$, so we picked boundaries
corresponding to $X$: $a=(i,j),b=(i-1,j+1)$ , and for $Y$: $a'=(i',j'),b'=(i'-1,j'+1)$
. Since $X$ is critical, we know $i\leq i'$, and so we only need
to show all $a_{t}$ between $a$ and $a'$ satisfy $i\leq depth(a_{t})$.
Since $X$ is critical, all vertex $v$ on $L$ of the dual graph
$\bar{H}'$ satisfy $depth(v)\geq i$. However, if a point $v\in\bar{H}$
is of depth $d$ is blue-based, then the corresponding $a_{t}$ will
have depth $d$ , whereas if it is red-based, then the corresponding
$a_{t}$ will have depth $d-1$. Therefore, for all $a_{t}$ to satisfy
$i\leq depth(a_{t})$ , we need all intermediate $v$ that have $depth(v)=i$
(i.e. $L$-critical points) to be blue-based. This is exactly the
condition assumed in the definition of good pairs. The situation is
shown on \figref{critical-points}.

Therefore, we have found $s+1$ positions on $L$ that satisfy \corref{k-points},
so we obtain a contradiction. It remains to show that we can pick
out $2s+1$ pairwise good critical points. 

We proceed by induction. Assume there is some $C_{k}$ such that any
chromatic boundary line with length at least $C_{k}$ we must be able
to find $k$ pairwise good critical points. For $k=1$ we pick any
critical point. 

Suppose there is a subsegment $L'$ of $L$ of length at least $C_{k}$
without any $L$-critical points. We pick out a set $S$ of $k$ pairwise
good critical points from $L'$. Let $v$ be the next immediate critical
point along $L$, to the either side of $L'$. Then $S+v$ is a pairwise
good collection of $k+1$ critical points. This is since $v$ is immediate,
there is no $L-$critical points between $v$ and any $w\in S$ along
$L$ . Therefore, every subsegments of length $C_{k}$ must contain
a critical point. Moreover, there cannot be $k+1$ consecutive critical
points of the same colour base on $L$ , since such $k+1$ points
are clearly pairwise good. From the two observations, from a critical
point we can find a critical point of opposite colour base by transversing
a distance of no more than $(k+1)C_{k}$ along $L$.

Let $u=(k+1)(2C_{k}+5)$. If $L$ is at least of length $u(k+1)C_{k}$
, then we can find critical points $A_{1},B_{1},A_{2}\cdots A_{u},B_{u}$
in this order along $L$, such that each $A_{i}$ is red based and
each $B_{i}$ is blue based, and that the next blue based critical
point after $A_{i}$ is $B_{i}$, and the next red based critical
point after $B_{i}$ is $A_{i+1}$. Now we consider the geometry of
$L$. Let $i_{0}$ denotes lowest depth. Then the entirety of $L$
is contained inside the half plane $P$ where depth $\geq i_{0}$.
Consider the continuous path $\gamma$ connecting $A_{\lceil\frac{u}{2}\rceil}$
and $B_{\lceil\frac{u}{2}\rceil}$, this divides the half plane $P$
into inside, and outside (see \figref{critical-curve} for the situation
and extra explanations). Therefore, either $A_{1}B_{1}\cdots B_{\lceil\frac{u}{2}\rceil}$
or $A_{\lceil\frac{u}{2}\rceil}B_{C_{k}+1}\cdots B_{2C_{k}+1}$ are
contained inside $\gamma.$ Regardless of which way it is, we can
infer that $A_{\lceil\frac{u}{2}\rceil}$ and $B_{\lceil\frac{u}{2}\rceil}$
are at least $\lceil\frac{u}{2}\rceil$ columns apart. This means
the length of path between $A_{\lceil\frac{u}{2}\rceil}$ and $B_{\lceil\frac{u}{2}\rceil}$
is at least $\lceil\frac{u}{2}\rceil$. However, the length of this
path cannot exceed $(k+1)C_{k}$ , since there cannot be any blue
based points between $A_{\lceil\frac{u}{2}\rceil}$, $B_{\lceil\frac{u}{2}\rceil}$,
cannot be more than $k$ red based critical points, and any subsegment
of length $C_{k}$ must have a critical point. As we have picked $\frac{u}{2}>(k+1)C_{k}$,
this is a contradiction. So $C_{k+1}=u(k+1)C_{k}$ completes the induction. 
\end{proof}
\begin{figure}
\centering{}\includegraphics[scale=0.5]{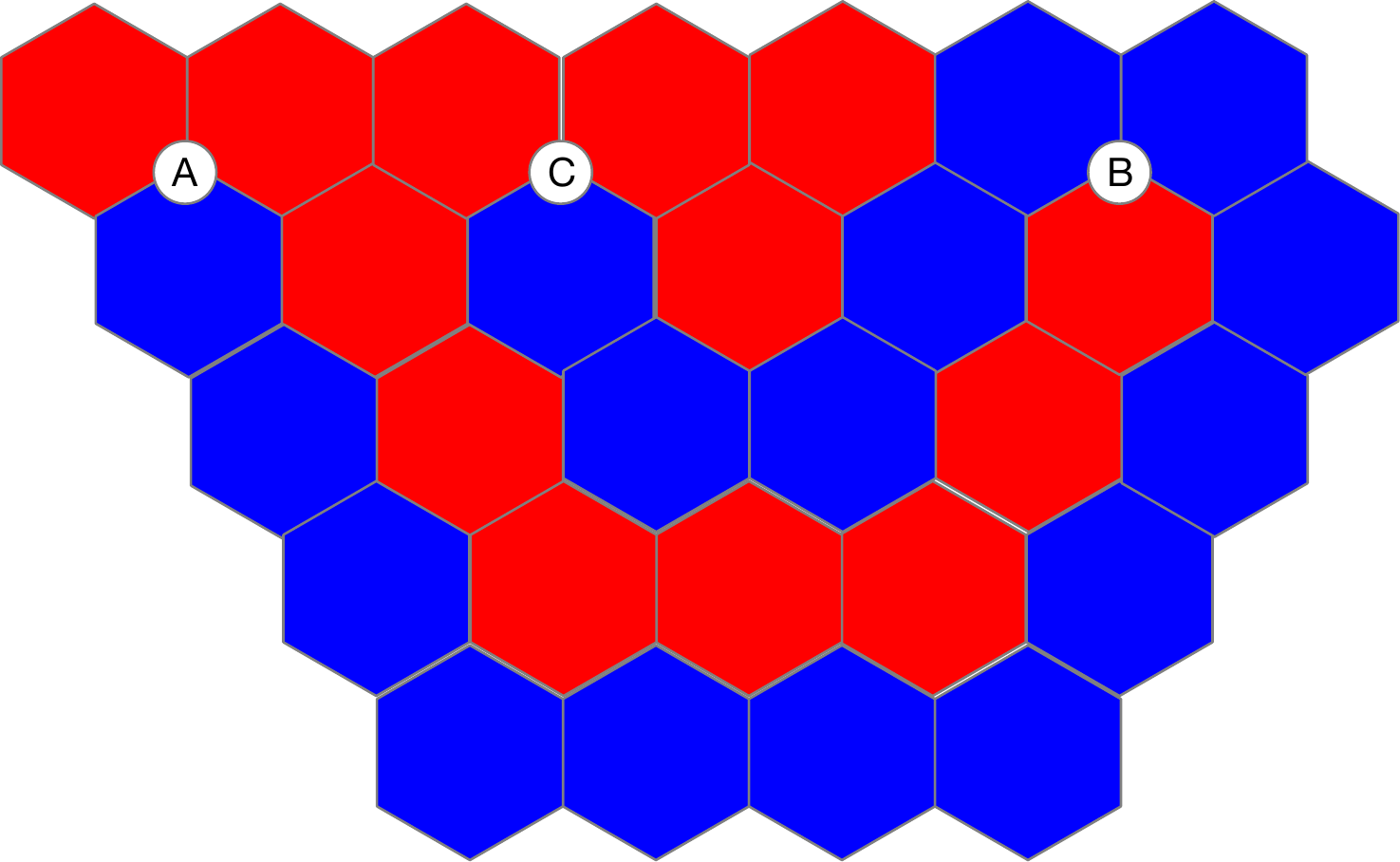}
\caption{{\small{}\label{fig:critical-points}Consider the segment of chromatic
boundary from point $A$ to point $C$. $A,B,C$ are all critical
points, $A$,$C$ are blue based and $B$ is red based. Since a critical
point of opposite base is present as we travel from $A$ to $C$,
we would have visited a blue hexagon of depth lower than $min(depth(A),depth(C))$,
so taking the two blue hexagons near point $A,C$ does not satisfy
the hypothesis of \corref{k-points}.}}
\end{figure}

\begin{figure}
\centering{}\includegraphics[scale=0.5]{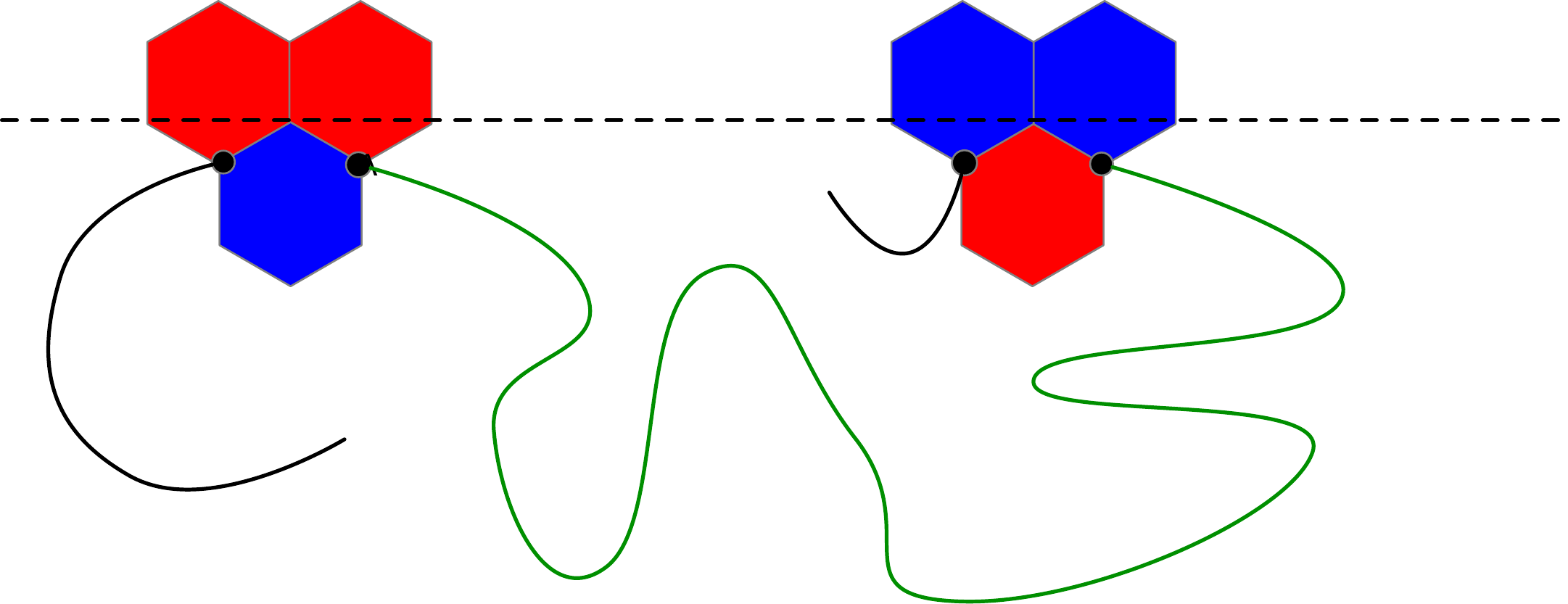}\caption{{\small{}\label{fig:critical-curve} If a boundary travels from a
critical point to a critical point of opposite base (the green curve),
it must divide the half-plane into two. Since the two points are assumed
critical, the boundary cannot travel to any point above the dotted
line. Since the colour blue must alway remain on the left or remain
on the right, the green curve must connect the critical points on
the same left or right side. }}
\end{figure}

\begin{lem}
\label{lem:top-or-long}For any sufficiently large hexagonal grid,
in any of its two colourings $\chi$ we can find either
\begin{itemize}
\item Some $s+1$ points $(1,k_{1}),(1,k_{2})\cdots(1,k_{s+1})$ that belong
to the same monochromatic component of $\chi$, Or,
\item A long chromatic boundary of length at least $S$ (where the constant
$S$ is from \lemref{long-line-bad})
\end{itemize}
\end{lem}

\begin{proof}
We prove this by contradiction and assume we cannot find a chromatic
boundary line of length at least $S$, and that there are no $s+1$
points of form $(1,i)$ that belongs to the same monochromatic component.

Let's call points $(1,x)$ such that $\chi(1,x)\neq\chi(1,x+1)$ cut
point. If $\chi(1,x)=\chi(1,x+1)=\cdots=\chi(1,x+s)$, then these
clearly forms $s+1$ points for the first condition to be satisfied.
So we can assume there is at least one cut point every $s+1$ columns.
Recall the property of hexagonal board that if $\chi(1,x)\neq\chi(1,x+1)$,
then there must be a chromatic boundary line starting from this point
$(1,x,-)$, and such a chromatic boundary line must also end at some
vertex on some boundary.

In the argument that follows, we will start from an arbitrary cut
point and move around the hexagonal board. We will not move more than
approximately $M=(s+2)S$ distance away from this point. Let's assume
the hexagonal board is of width at least $2M+2S$ , and of depth least
$S$. We will also pick the starting cut point to be at the middle,
i.e. $(1,x)$ for $x$ about $M+S$. Under these assumptions, a cut
point that starts from an initial point must also end in a cut point,
for it is too far away from the other three sides, so ending at the
other sides would imply the existence of a length $S$ boundary. 

For a boundary that starts from cut point $(1,x,-)$ and ends at cut
point $(1,y,-)$, we write this boundary as $B(x,y)$ . We always
assume $x<y$ in the notation of $B(x,y)$. Clearly, since chromatic
boundary line is a continuous curve on the dual graph $\bar{H}'$,
boundary lines cannot cross. That is, there cannot exist $B(x,y)$
and $B(x',y')$ such that $x<x'<y<y'$. We say boundary $B(x,y)$
contains $B(x',y')$ if $x<x'<y'<y$. A boundary line $B(x,y)$ is
maximal, if it is not contained by any boundary line. Note that every
boundary line is either itself maximal, or is contained by a maximal
boundary. This is since we can keep replacing the boundary line by
any that contains it, but the length of a boundary line cannot exceed
$S$ so this process must stop. By assumption, all maximal boundary
lines have length at most $S$. 
\begin{claim*}
Let $B(x,y)$ be maximal boundary. Let $a$ be the smallest value
such that $a>y$ and $(1,a)$ is a cut point, then if $B(a,b)$ is
the boundary that $a$ belongs to , we must have $a<b$ and $B(a,b)$
maximal. 
\end{claim*}
\begin{proof}
Let the boundary that starts at $a$ ends at cut point $(1,b)$. There
are a few choices for value $b$, either $b<x$ , $x<b<y$, $y<b<a$
or $a<b$. If $b<x$, then $B(b,a)$ contains $B(x,y)$, contradicting
maximality of $B(x,y)$. If $x<b<y$, then $B(b,a)$ crosses with
$B(a,b)$. If $y<b<a$ , then this contradicts the choice of $a$.
Therefore, $a<b$. 

To see that $B(a,b)$ is maximal, let's assume it is contained by
$B(a',b')$, then either $a'<x$, $x<a'<y$ or $y<a'<a$. Similar
to the above, the first case would have $B(a',b')$ contain $B(x,y)$,
the second case would have $B(a',b')$ cross with $B(x,y)$ , the
third case would contradicts choice of $a$. So $B(a,b)$ must be
maximal.
\end{proof}
To complete the proof, we start with a cut point $C$ and find a nearby
maximal boundary. Suppose this cut point belong to boundary $B(x,y)$,
we find the maximal boundary containing $B(x,y)$ and define this
as $B(x_{1},y_{1})$. Note that we have not moved away from $C$ by
a distance of more than $S$. Now, apply the claim $s$ times, we
can find a chain of maximal boundaries $B(x_{1},y_{1})$, $B(x_{2},y_{2})\cdots B(x_{s+1},y_{s+1})$,
with the additional property that the next cut point after $y_{i}$
is $x_{i+1}$. For each $B(x_{i},y_{i})$, observe that we must have
$\chi(1,x_{i})=\chi(1,y_{i}+1)$, and in the same monochromatic component.
This is due to geometry, in that the colour $\chi(1,x_{i})$ starts
on the right of the boundary, so must end at the right of the boundary,
which is $\chi(1,y_{i}+1)$. By traveling along the boundary line
$B(x_{i},y_{i})$, we can see that $(1,x_{i})$ and $(1,y_{i}+1)$
belongs to the same monochromatic component. Moreover, since $y_{i}$
is the next cut point after $x_{i+1}$, we have $\chi(1,x_{i+1})=\chi(1,y_{i}+1)$
and they are in the same monochromatic component. In conclusion, $\chi(1,x_{1})=\chi(1,x_{2})\cdots=\chi(1,x_{s+1})$,
and they are all in the same monochromatic component. Therefore, we
have found $s+1$ points as required.
\begin{figure}[h]
\centering{}\includegraphics{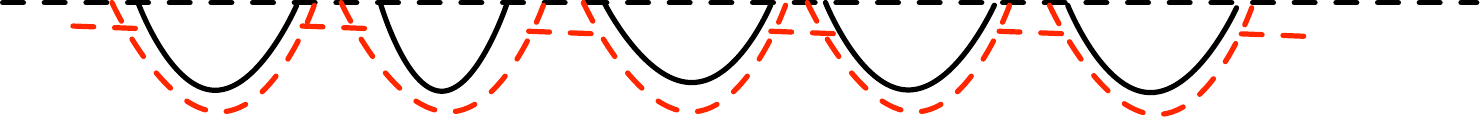}\caption{Picking out $s+1$ maximal boundaries. The dotted red line showed
the side forms a monochromatic component}
\end{figure}
\end{proof}
The proof of the main theorem is now complete 
\begin{proof}
Given a $s$-stack layout of $T\boxslash P$ , start by performing
the three passes to subtree as described in \secref{ramsey-passes}.
We assume $T\boxslash P$ is large enough to begin with such that
\ref{lem:top-or-long} can be applied after all the subtree passes.
We have obtained two colours on multiple hexagonal grid $Z(i,j,p)$.
Apply \ref{lem:top-or-long} to the hexagonal board $Z(1,*,*)$. If
we find a chromatic boundary line of length at least $S$, then contradiction
follows from \ref{lem:long-line-bad}. If we have found $s+1$ points
belonging to the same monochromatic component, contradiction follows
from \ref{cor:k-points-top}. Either way, we have reached a contradiction
and the proof is complete.
\end{proof}

\section{$T\boxslash P$ has queue number at most $3$ \label{sec:queue-3}}

It is not difficult to show $T\boxslash P$ has queue number at most
$3$. The three types of edges can be put into their own queue. We
shall write this out in our current notations. However, it is not
clear if $T\boxslash P$ has a $2$-queue layout. We conjecture the
answer to be ``NO'' if $T\boxslash P$ is sufficiently large.
\begin{prop}
$T\boxslash P$ has queue number at most 3
\end{prop}

\begin{proof}
We will first define the total order imposed on vertices of $T\boxslash P$.
The total order compares two elements as follows: For two elements
$(A,i)$, $(B,j)$ :
\begin{itemize}
\item If $i\neq j$ , $(A,i)<(B,j)$ iff $i<j$. Else, 
\item If $|A|\neq|B|$, $(A,i)<(B,j)$ iff $|A|<|B|$. Else,
\item $(A,i)\leq(B,j)$ iff $A$ is lexicographically smaller than $B$.
\end{itemize}
Since we only need to compare $A$ and $B$ lexicographically when
$|A|=|B|$ , the third condition is equivalent to $A_{d}<B_{d}$,
where $d$ is the first position where $A_{d}\neq B_{d}$.

Now, we will use one queue for each of the three types of edges in
definition of $T\boxslash P$. We claim that no edges of the same
type nests.

To show that no edges nests, it suffices to show that either 
\begin{enumerate}
\item For any edges $(a,b)$ and $(c,d)$ where $a<b$ , $c<d$ and assume
in addition $a<c$, then we must have $b<d$. or 
\item For any edges $(a,b)$ and $(c,d)$ where $a<b$ , $c<d$ and assume
in addition $b<d$, then we must have $a<c$. 
\end{enumerate}
This is exactly the concept of ``first in first out'' about queues
as a data structure. To see this is sufficient, given any pairs of
edges $(a,b)$ and $(c,d)$, without loss of generality we may assume
$a<b$, $c<d$ and $a<c$. In this case, the only possible order of
the four vertices are $a<b<c<d$ , $a<c<b<d$, $a<c<d<b$. Only the
third option is nesting. whereas $b<d$ means the order must be the
first two types. The second case is analogous.

\subsubsection*{Vertical Edges}

Suppose $(v,i)<(w,j)$ , we need to show that $(v+[a],i)<(w+[b],j)$
for any integers $a,b$ . We consider the condition that results in
$(v,i)<(w,j)$:

If $i<j$ , then we have $(v+[a],i)<(w+[b],j)$ directly.

If instead $i=j$, now if $|v|<|w|$ , then $|v+[a]|<|w+[b]|$, so,
and so $(v+[a],i)<(w+[b],j)$,

If instead $|v|=|w|$, we must have $v<_{lex}w$ , so $(v+[a])$ is
lexicographically smaller than $(w+[b])$, so$(v+[a],i)<(w+[b],j)$
.

\subsubsection*{Horizontal Edges}

Suppose $(v,i)<(w,j)$ , we need to show $(v,i+1)<(w,j+1)$.

If $i<j$, then $i+1<j+1$ , so $(v,i+1)<(w,j+1)$.

If instead $i=j$, then either $|v|<|w|$ or $v<_{lex}w$. In both
cases, it is easy to see that $(v,i+1)<(w,j+1)$.

\subsubsection*{Diagonal Edges}

Suppose $(v,i)<(w,j)$ , we need to show $(v+[a],i+1)<(w+[b],j+1)$

If $i<j$, then $i+1<j+1$ so $(v+[a],i+1)<(w+[b],j+1)$ follows.

If instead $i=j$, then $i+1=j+1$. In this case, the same argument
from vertical edges apply.

Therefore, we have checked directly that three queues are enough for
graph $T\boxslash P$.
\end{proof}

\section{Open Problems}

From \cite{key-first-4}, we know that stack number is bounded by
queue number. The other question is still open: 
\begin{problem}
Is queue number bounded by stack number?
\end{problem}

It is known that graphs with queue number $1$ have a $2$-stack layout
(\cite{1-queue}). We have shown that graphs with queue number $3$
may have unbounded stack number. This leaves the case of queue number
$2$ open.
\begin{problem}
Do graphs with queue number $2$ have unbounded stack number?
\end{problem}


\begin{thebibliography}{1}
\bibitem{key-first-4}Dujmovi\'{c}, V., Eppstein, D., Hickingbotham,
R. et al. Stack-Number is Not Bounded by Queue-Number. Combinatorica
42, 151--164 (2022). https://doi.org/10.1007/s00493-021-4585-7

\bibitem{three-products}Eppstein, David, et al. \textquotedbl Three-dimensional
graph products with unbounded stack-number.\textquotedbl{} Discrete
\& Computational Geometry (2023): 1-28.

\bibitem{erdosgoodbound}Buci\'{c}, Matija, Benny Sudakov, and Tuan
Tran. \textquotedbl Erd\H{o}s--Szekeres theorem for multidimensional
arrays.\textquotedbl{} Journal of the European Mathematical Society
(2022).

\bibitem{erdos}P.C Fishburn, R.L Graham, Lexicographic Ramsey theory,
Journal of Combinatorial Theory, Series A, Volume 62, Issue 2, 1993,
Pages 280-298, ISSN 0097-3165, https://doi.org/10.1016/0097-3165(93)90049-E.

\bibitem{hex}Gale, David. \textquotedblleft The Game of Hex and the
Brouwer Fixed-Point Theorem.\textquotedblright{} The American Mathematical
Monthly, vol. 86, no. 10, 1979, pp. 818--27. JSTOR, https://doi.org/10.2307/2320146.
Accessed 20 Mar. 2023. 

\bibitem{1-queue}Heath, Lenwood S., Frank Thomson Leighton, and Arnold
L. Rosenberg. \textquotedbl Comparing queues and stacks as machines
for laying out graphs.\textquotedbl{} SIAM journal on discrete mathematics
5.3 (1992): 398-412.
\end{thebibliography}
\end{document}